\documentclass[reqno]{amsart}

\usepackage{a4wide}
\usepackage{color}
\usepackage{mathrsfs}
\usepackage{mathtools}
\usepackage{tikz-cd}
\usepackage{amsmath}
\usepackage{amssymb}
\numberwithin{equation}{section}
\usepackage[colorlinks,citecolor=green,linkcolor=red]{hyperref}
\usepackage{hyphenat}
\usepackage{enumitem}

\usepackage[latin1]{inputenc}

\newcommand{\N}{\mathbb{N}}
\newcommand{\R}{\mathbb{R}}
\newcommand{\sfd}{{\sf d}}
\renewcommand{\d}{{\mathrm d}}
\newcommand{\eps}{\varepsilon}
\newcommand{\fr}{\penalty-20\null\hfill\(\blacksquare\)}
\newcommand{\X}{{\rm X}}
\newcommand{\Y}{{\rm Y}}
\newcommand{\sus}{\subseteq}

\newtheorem{theorem}{Theorem}[section]
\newtheorem{corollary}[theorem]{Corollary}
\newtheorem{lemma}[theorem]{Lemma}
\newtheorem{proposition}[theorem]{Proposition}
\newtheorem{definition}[theorem]{Definition}

\newtheorem{remark}[theorem]{Remark}

\title{Direct limits of infinite-dimensional Carnot groups}
\author{Terhi Moisala}
\address{University of Jyv\"{a}skyl\"{a}}
\email{moisala.terhi@gmail.com}

\author{Enrico Pasqualetto}
\address{Scuola Normale Superiore, Piazza dei Cavalieri 7,
56126 Pisa, Italy}
\email{enrico.pasqualetto@sns.it}

\begin{document}
\date{January 11, 2021} 
\begin{abstract}
We give a construction of direct limits in the category of complete metric scalable groups and provide sufficient conditions for the limit to be an infinite-dimensional Carnot group. We also prove a Rademacher-type theorem for such limits.
\end{abstract}
\maketitle
\tableofcontents
\section{Introduction}
\subsection{Overview}
The main purpose of this paper is to study direct limits of Carnot groups. Carnot groups are characterized as the locally compact, geodesic metric spaces, that are isometrically homogeneous and admit dilations
(see \cite{LD13}). During the last decades, they have provided a fruitful framework for Geometric Analysis and Metric Geometry (see \cite{LD17} for a comprehensive introduction to Carnot groups). Recently, there has been growing interest in generalizing Carnot groups into infinite dimensions. A notion of infinite-dimensional Heisenberg group based on an abstract Wiener space was introduced in \cite{DriverGordina08, BGT13} and Lie groups generalizing those in \cite{Mel09}. A form of infinite-dimensional sub-Riemannian geometry from control theoretic viewpoint was suggested in \cite{GMV15}. In \cite{mr13,MPS19}, a Rademacher-type theorem has been proven when the target is a so called Banach homogeneous group, which is a Banach space equipped with a suitable non-abelian group structure. In a recent paper \cite{LDZ19}, the authors study inverse limits of free nilpotent Lie groups.

By studying direct limits of Carnot groups we continue, from a constructive viewpoint, the development of infinite-dimensional generalization of Carnot groups.
Our starting point is the notion of infinite-dimensional Carnot group introduced by E.\ Le Donne, S.\ Li, and the first named author in \cite{LDLM19}. The direct limits are considered -- in the categorical sense -- in the context of \emph{(complete) metric scalable groups} (briefly, a (C)MSG). Metric scalable groups are metric groups endowed with a family \(\{\delta_\lambda\}_{\lambda\in\R}\subseteq{\rm Aut}(G)\) of \emph{dilations} of \(G\) (see Definition \ref{def:metric_scalable_gp}). By an infinite-dimensional Carnot group \(G\) we mean a CMSG admitting a \emph{filtration by Carnot subgroups}, meaning that there exists an increasing sequence \(N_1\subseteq N_2\subseteq\ldots\subseteq G\) of Carnot groups whose union is dense in \(G\); the precise definition is recalled in Definition \ref{def:inf-dim_Carnot}.

It is rather straightforward to show that every (infinite-dimensional) Carnot group is a direct limit of finite-dimensional Carnot groups (see Proposition \ref{prop:infdim_Carnot_is_DL}) in the category of CMSGs. A more challenging task is to understand when a direct system of infinite-dimensional Carnot groups has a direct limit in the category of CMSGs, and when the limit is an infinite-dimensional Carnot group itself. We give now an overview of our results.

\subsection{Construction of direct limits}
We start our consideration from scalable groups, that in this paper are groups equipped with dilation automorphisms. It turns out (not surprisingly) that direct limits of scalable groups always exist, since they are purely algebraic objects and no topology is involved. Their existence will be proved
(via an explicit construction) in Theorem \ref{thm:DL_scalable_gps}.

In the setting of MSGs the situation is much more delicate, as we are now going to describe. In analogy with the case of Banach spaces,
one might heuristically expect that -- calling
\(\big(G',\{\varphi'_i\}_{i\in I}\big)\) the direct limit of
our direct system in the category of scalable groups --
the direct limit as a MSG can be built as follows:
first, we define the pseudodistance \(\sfd'\) (i.e., distinct
points may have zero distance) on \(G'\) as
\[
\sfd'(x,y)\coloneqq\inf\Big\{\sfd_i(x_i,y_i)\;\Big|\;
i\in I,\,x_i,y_i\in G_i,\,\varphi'_i(x_i)=x,\,\varphi'_i(y_i)=y\Big\}
\quad\text{ for every }x,y\in G';
\]
second, we consider the quotient \(G'/\sim\), where \(x\sim y\) if and
only if \(\sfd'(x,y)=0\), which inherits a natural structure of metric
space. This sort of construction will be investigated
in Section \ref{ss:DL_MSG} and is called the \emph{metric limit} (see Definition \ref{def:metric_limit}). In the case of CMSGs, the candidate direct limit \(G\) is given by the metric completion \((G,\sfd)\) of the metric limit \(G'/\sim\).

In Banach spaces all the vector space operations are automatically Lipschitz-continuous. In fact, direct limits of Banach spaces always exist and are exactly the completions of metric limits. However, in the context of (C)MSGs, the metric limit does not necessarily provide us with a (C)MSG. The obstruction comes
from the fact that there is no reason why right translations, inversion,
and dilation on \(G'\) should be continuous with respect to \(\sfd'\);
only left translations must be continuous due to
the left-invariance of \(\sfd'\). Consequently, one cannot expect
the scalable group structure to carry over to  \(G'/\sim\) and to
its completion \(G\), and even if it does, the resulting space \(G\)
does not need to be a topological group. We give an example of this troublesome phenomenon in Section \ref{ss:ex_deg_DS}.

In Definitions \ref{def:non-deg_DS_MSG} and \ref{def:non-deg_DS_CMSG} we introduce some conditions on the continuity of the operations on $ G'$ that guarantee the well-posedness of the metric limit and its metric completion. When the direct system in the category of MSGs or CMSGs satisfies these conditions, we say that it is \emph{non-degenerate} in the corresponding category. Our main result is the following:
\begin{theorem}[Direct limits of (C)MSGs]\label{thm:DL_equiv_non-deg}
	Let $\big(\{G_i\}_{i\in I},\{\varphi_{ij}\}_{i\leq j}\big)$ be a direct system of (complete) metric scalable groups. Then the direct limit exists and equals the (completion of the) metric limit if and only if the direct system is non-degenerate.
\end{theorem}

The non-degeneracy assumptions can be formulated as a quantitative condition
on the direct system (see Propositions \ref{prop:char_non-def_DS} and \ref{prop:char_non-def_DS_complete}). In the general case the formulation is quite involved, but in some specific circumstances of interest many simplifications occur; for instance, when the elements of the direct system have uniformly bounded nilpotency step (Corollary \ref{cor:nilpotent_DL}) or have a Carnot group structure (Theorem \ref{thm:DL_of_infdim_Carnots_proof}). We remark that these kinds of questions on continuity of group operations have also been of independent interest, see e.g.\ \cite{Mon36} for a classical result and \cite{Tkacenko14} for a survey on semitopological groups.

 We stress that the non-degeneracy condition is necessary and sufficient for the fact that the metric completion of the metric limit is a metric scalable group, but it is not clear if this fact is equivalent to the existence of the direct limit.  A problem that still remains open is the following: 
\begin{itemize}
	\item[\({\rm P1})\)] Is there a direct system of CMSGs
	that is degenerate, whose direct limit exists?
\end{itemize}
\medskip

Regarding the question if infinite-dimensional Carnot groups are stable under the operation of taking direct limits, we provide a sufficient criterion. In this paper we say that a direct system is countable if the indexing set $ I $ is countable.
\begin{theorem}[Direct limits of infinite-dimensional Carnot groups]
\label{thm:DL_of_infdim_Carnots}
	Let $\big(\{G_i\}_{i\in I},\{\varphi_{ij}\}_{i\leq j}\big)$ be a countable non-degenerate direct system of infinite-dimensional Carnot groups in the category of CMSGs. If each group $ G_i $ is nilpotent, then the direct limit is an infinite-dimensional Carnot group.
\end{theorem}
The complete picture of the relation between nilpotency of the groups $ G_i $ and non-degeneracy of the direct system remains unclear; we record the following open problems.
\begin{itemize}
	\item[\({\rm P2})\)] Can we remove the nilpotency assumption in the statement of Theorem
	\ref{thm:DL_of_infdim_Carnots}?
	\item[\({\rm P3})\)] If $\big(\{G_i\}_{i\in I},\{\varphi_{ij}\}_{i\leq j}\big)$ is a direct system of infinite-dimensional Carnot groups such that the nilpotency steps of the groups $ G_i $ are uniformly bounded, is it non-degenerate?
\end{itemize}
\medskip
Let us also mention that inverse limits of CMSGs can be treated
in a similar way. Indeed, inverse limits of scalable groups always exist (see Theorem \ref{thm:IL_scalable_gps}), but at the level of CMSGs, the expected construction of an inverse limit might fail to work. We characterize
when inverse limits exist and are of the desired form (Theorem \ref{thm:IL_MSG}) in terms of a suitable non-degeneracy condition (Definition \ref{def:non-deg_IS}). Appendix \ref{ss:IL_CMSG} will be dedicated to inverse limits of CMSGs,
but we do not insist further on them, as in this paper we will
not provide any of their possible applications. For the same reason, we leave the study of general limits and colimits in this category for future research.
\subsection{A Rademacher-type theorem}
We study the stability of CMSGs satisfying a Rademacher theorem under taking direct limits.  One of the main results that have been achieved in \cite{LDLM19}, concerning infinite-dimensional Carnot groups, is a variant of the classical Rademacher's theorem on the almost everywhere differentiability of Lipschitz functions. Roughly speaking, the result says that each Lipschitz function on a given infinite-dimensional Carnot group is G\^{a}teaux differentiable at \emph{almost every} point, where the notion of `negligible set' is expressed in terms of the filtration by Carnot subgroups. A key ingredient in the proof is the celebrated Pansu--Rademacher theorem for Lipschitz functions on Carnot groups \cite{Pansu89}.

In this paper, we say that a couple \((G,\mathcal N)\) -- where \(G\) is a CMSG and \(\mathcal N\) is a \emph{\(\sigma\)-ideal of null sets}, in the sense of
Definition \ref{def:family_null_sets} -- has the \emph{Rademacher property} if for any Lipschitz function \(f\colon G\to\R\) the set of points where \(f\) is
not G\^{a}teaux differentiable (in the sense of Definition \ref{def:Gateaux_diff}) belongs to \(\mathcal N\); see Definition \ref{def:Rademacher_prop}. The main result regarding the Rademacher property reads as follows:
\begin{theorem}[Rademacher theorem for direct limits]
\label{thm:Rademacher_property}
Let \(\big\{(G_i,\mathcal N_i)\big\}_{i\in I}\) be a countable direct system  of complete metric scalable groups  having the Rademacher property. If the system admits a direct limit $ G $, then $ (G,\mathcal N) $ has the Rademacher property for a natural \(\sigma\)-ideal \(\mathcal N\) of null sets
(depending on the direct system).
\end{theorem}
Theorem \ref{thm:Rademacher_property} generalizes the Rademacher's theorem for infinite-dimensional Carnot groups proven in \cite{LDLM19}. Indeed, by Proposition \ref{prop:infdim_Carnot_is_DL}, if $ G $ is an infinite-dimensional Carnot group with a filtration $ (G_i)_{i\in \N} $ by Carnot groups, then $ G $ is obtained as the direct limit of $ (G_i)_{i\in \N}  $. The theorem in \cite{LDLM19} is recovered by choosing $ (\mathcal N_i)_{i\in\N} $ to be the $ \sigma $-ideals of null sets of some Haar measures on the Carnot groups $ G_i $, where the Rademacher property of the pairs $ (G_i,\mathcal N_i) $ is given by \cite{Pansu89}.

Our definition of the \(\sigma\)-ideal \(\mathcal N\) on
the direct limit \(G\) closely follows along the construction
presented in \cite{LDLM19}, which is in turn inspired by what done
in \cite{Aronszajn1976}.
However, while in \cite{Aronszajn1976,LDLM19} it is shown
that \(\mathcal N\) is non-trivial (in the sense that its elements
have empty interior, so in particular every Lipschitz function has
at least one differentiability point), it seems that -- at our level
of generality -- nothing about \(\mathcal N\) can be said. It would be interesting to find other sufficient conditions
for the \(\sigma\)-ideal \(\mathcal N\) to be non-trivial.

\subsection{Structure of the paper}
In Sections \ref{ss:basic_notions} and \ref{ss:DL_IL} we recall some results on Cauchy-continuity and the notions of direct and inverse limits in the categorical sense, respectively. The existence of limits in the category of scalable groups is proven in Section \ref{ss:scalable_gps}. In Section \ref{ss:sc_gps_dist} we prove some auxiliary continuity results on scalable groups with compatible distances. Sections \ref{ss:DL_MSG} and \ref{ss:DL_CMSG} are devoted to the direct limits of MSGs and CMSGs, respectively. Direct limits of infinite-dimensional Carnot groups are discussed in Section \ref{s:inf-dim_Carnots}.
The problem of stability of CMSGs satisfying a Rademacher theorem under taking direct limits will be addressed in Section \ref{s:Rademacher}. Finally, in Appendix \ref{ss:IL_CMSG} we briefly investigate the existence of inverse limits of (C)MSGs.
\section{Preliminaries}
\label{s:prelis}
\subsection{Basic notions in metric geometry}
\label{ss:basic_notions}
Let \((I,\leq)\) be a \emph{directed set}, meaning that \(\leq\) is
a preorder (i.e., a reflexive and transitive binary relation)
on \(I\) with the property that any two elements have an upper bound
(i.e., for any \(i,j\in I\) there exists \(k\in I\) such that \(i\leq k\)
and \(j\leq k\)).

A \emph{net} indexed over \(I\) in a given set \(\X\) is any map
\(x\colon I\to\X\). We shall often denote \(x\) by \(\{x_i\}_{i\in I}\).
If \(\X\) is a topological space, then the net \(\{x_i\}_{i\in I}\) is said
to \emph{converge} to a point \(\bar x\in\X\) provided for every neighborhood
\(U\) of \(\bar x\) there exists \(i_0\in I\) such that \(x_i\in U\) for every
\(i\in I\) such that \(i\geq i_0\). In such case, we write
\(\bar x=\lim_{i\in I}x_i\).
\begin{definition}[Cauchy-continuity]
Let \((\X,\sfd_\X)\), \((\Y,\sfd_\Y)\) be pseudometric spaces.
Let \(\varphi\colon\X\to\Y\) be continuous. Then we say
that \(\varphi\) is \emph{Cauchy-continuous} provided it
satisfies the following property:
\[
(x_n)_n\subseteq\X\text{ is Cauchy }\quad\Longrightarrow\quad
\big(\varphi(x_n)\big)_n\subseteq\Y\text{ is Cauchy.}
\]
\end{definition}
By a \emph{metric completion} of a given metric space \((\X,\sfd)\)
we mean any couple \((\bar\X,\iota)\), where \((\bar\X,\bar\sfd)\) is
a complete metric space, while \(\iota\colon\X\hookrightarrow\bar\X\)
is an isometric embedding having dense image. It holds that \((\bar\X,\iota)\)
is uniquely determined up to unique isomorphism, meaning that for any other
couple \((\bar\X',\iota')\) satisfying the same property there exists a unique
isometric bijection \(\Phi\colon\bar\X\to\bar\X'\) such that \(\Phi\circ\iota=\iota'\).
Occasionally, we will implicitly identify \(\X\) with the subspace \(\iota(\X)\)
of \(\bar\X\), and we will just say that \(\bar\X\) is \emph{the} metric completion
of \(\X\) (without mentioning the embedding \(\iota\)).
\medskip

Let us recall the following well-known fact. For the reader's convenience,
we sketch its proof.
\begin{theorem}[Extension theorem]\label{thm:extension_to_completion}
Let \((\X,\sfd_\X)\), \((\Y,\sfd_\Y)\) be metric spaces. Suppose
\(\Y\) is complete. Let \(\varphi\colon\X\to\Y\) be a continuous map.
Denote by \(\bar\X\) the metric completion of \(\X\).
Then \(\varphi\) admits a (unique) continuous
extension \(\bar\varphi\colon\bar\X\to\Y\) if
and only if it is Cauchy-continuous.
\end{theorem}
\begin{proof}
Let us denote by \(\iota\colon\X\hookrightarrow\bar\X\) the isometric embedding
that comes with the metric completion.\\
{\color{blue}\textsc{Necessity.}}
Suppose \(\varphi\) admits a continuous extension \(\bar\varphi\colon\bar\X\to\Y\).
Let \((x_n)_n\) be a Cauchy sequence in \(\X\). Then there exists \(\bar x\in\bar\X\)
such that \(\iota(x_n)\to\bar x\). By using the continuity of the map \(\bar\varphi\),
we deduce that \(\varphi(x_n)=\bar\varphi\big(\iota(x_n)\big)\to\bar\varphi(\bar x)\).
In particular, \(\big(\varphi(x_n)\big)_n\) is a Cauchy sequence in \(\Y\).\\
{\color{blue}\textsc{Sufficiency.}} Suppose \(\varphi\) is Cauchy-continuous. We define
the map \(\bar\varphi\colon\bar\X\to\Y\) as follows: given any \(\bar x\in\bar\X\),
we define \(\bar\varphi(\bar x)\in\Y\) as the limit of \(\varphi(x_n)\) as
\(n\to\infty\), where \((x_n)_n\subseteq\X\) is any sequence satisfying
\(\iota(x_n)\to\bar x\). Notice that such a sequence \((x_n)_n\) exists (by density
of \(\iota(\X)\) in \(\bar\X\)) and is Cauchy, the sequence
\(\big(\varphi(x_n)\big)_n\subseteq\Y\) is Cauchy (thanks to the Cauchy-continuity
of \(\varphi\)), and the limit \(\lim_n\varphi(x_n)\in\Y\) exists because of the
completeness of \(\Y\). The well-posedness of \(\bar\varphi\) can be
easily checked, while its continuity follows from a standard diagonalization argument.
To show that the map \(\bar\varphi\) extends \(\varphi\), just consider for any point
\(\bar x=\iota(x)\in\iota(\X)\) the sequence \((x_n)_n\) that is constantly equal to
\(x\). Finally, the uniqueness of \(\bar\varphi\) is granted by the density of
\(\iota(\X)\) in \(\bar\X\).
\end{proof}
\begin{definition}[Modulus of continuity]\label{def:modulus_cont}
Let \((\X,\sfd_\X)\) and \((\Y,\sfd_\Y)\) be two pseudometric spaces.
Fix any map \(\varphi\colon\X\to\Y\). Then for any \(x\in\X\) and
\(\eps>0\) we define the quantity \(\omega_\varphi(x;\eps)\in[0,+\infty]\) as
\[
\omega_\varphi(x;\eps)\coloneqq\sup\Big\{\eta>0\;\Big|\;
\sfd_\Y\big(\varphi(x),\varphi(x')\big)<\eps\,\text{ for every }
x'\in\X\text{ with }\sfd_\X(x,x')<\eta\Big\}.
\]
The function \(\omega_\varphi(x;\cdot)\colon(0,+\infty)\to[0,+\infty]\) is
said to be the \emph{modulus of continuity} of \(\varphi\) at \(x\).
\end{definition}
Notice that \(\varphi\) is continuous at \(x\) if and only if
\(\omega_\varphi(x;\eps)>0\) for every \(\eps>0\). Similarly,
\(\varphi\) is uniformly continuous on a set \(B\subseteq\X\) if and only
if \(\inf_{x\in B}\omega_\varphi(x;\eps)>0\) holds for every \(\eps>0\).
\medskip

Recall that a subset \(T\) of a pseudometric space \((\X,\sfd)\) is said
to be \emph{totally bounded} provided for any radius \(r>0\) we can find
finitely many points \(x_1,\ldots,x_n\in T\) such that
\(T\subseteq\bigcup_{i=1}^n B_r(x_i)\).
\begin{lemma}[Characterizations of Cauchy-continuity]
\label{lemma:Cauchy_continuity_vs_moduli}
Let \((\X,\sfd_\X)\), \((\Y,\sfd_\Y)\) be pseudometric spaces and let  \(\varphi\colon\X\to\Y\) be a map.
Then the following are equivalent. 
\begin{enumerate}[label=\rm\roman{*})]
	\item \label{it:Cauchy} The map $ \varphi $ is Cauchy-continuous;
	\item \label{it:unif_cont} $ \varphi $ is uniformly continuous on each totally bounded subset of \(\X\);
\item \label{it:f(T)}$ \varphi $ maps all totally bounded subsets of $ \X $ to totally bounded subsets of $ \Y $.
\end{enumerate}
\end{lemma}
\begin{proof}
\ref{it:Cauchy}$ \implies $ \ref{it:unif_cont}: Suppose \(\varphi\) is Cauchy-continuous.
We argue by contradiction: assume there exists a totally bounded set
\(T\subseteq\X\) whereon \(\varphi\) is not uniformly continuous,
so that \(\inf_{x\in T}\omega_\varphi(x;\eps)=0\) for some \(\eps>0\).
This means that we can find a sequence \((x_{2n})_n\subseteq T\) such
that \(\lim_n\omega_\varphi(x_{2n};\eps)=0\). Since \(T\) is totally
bounded, we can assume (up to taking a not-relabeled subsequence)
that \((x_{2n})_n\) is Cauchy. Given any \(n\in\N\), we deduce from the
very definition of \(\omega_\varphi(x_{2n};\eps)\) that there is an element
\(x_{2n+1}\in\X\) such that \(\sfd_\X(x_{2n},x_{2n+1})<\omega_\varphi(x_{2n};\eps)+1/n\)
and \(\sfd_\Y\big(\varphi(x_{2n}),\varphi(x_{2n+1})\big)\geq\eps\).
Therefore, \((x_n)_n\) is a Cauchy sequence but \(\big(\varphi(x_n)\big)_n\)
is not, thus contradicting the assumption that \(\varphi\) is Cauchy-continuous.
Then \(\varphi\) is uniformly continuous on each totally bounded set.\\
\ref{it:unif_cont}$ \implies $ \ref{it:f(T)}: Let $ \eps > 0 $ and let $ T\sus \X $ be totally bounded. Pick $ \delta > 0 $ such that, for all $ x\in T $, $ \sfd_\Y\big(\varphi(x),\varphi(y)\big)<\eps $ whenever $ \sfd_\X(x,y)<\delta $. Pick $ x_1,\dots,x_n \in T $ such that $ T \sus \bigcup_{i=1}^n B_\delta(x_i) $. Then $ \varphi(T)\sus \bigcup_{i=1}^n B_\eps(\varphi(x_i)) $, proving that $ \varphi(T) $ is totally bounded by the arbitrariness of $ \eps $.\\
 \ref{it:f(T)}$ \implies $ \ref{it:Cauchy}: The claim follows from the observation that a sequence $ (x_n)_n \sus \X $ is Cauchy if and only if it is totally bounded as a subset of $ \X $.
\end{proof}
Given two pseudometric spaces \((\X,\sfd_\X)\) and \((\Y,\sfd_\Y)\),
we always endow their product \(\X\times\Y\) with the pseudodistance
\(\sfd_{\X\times\Y}\), which is defined as
\[
\sfd_{\X\times\Y}\big((x,y),(x',y')\big)\coloneqq
\sqrt{\sfd_\X(x,x')^2+\sfd_\Y(y,y')^2}\quad\text{ for every }
(x,y),(x',y')\in\X\times\Y.
\]
If both \(\sfd_\X\) and \(\sfd_\Y\) are distances, then \(\sfd_{\X\times\Y}\)
is a distance as well. Observe also that the topology induced by
\(\sfd_{\X\times\Y}\) on \(\X\times\Y\) coincides with the product of
the topologies induced by \(\sfd_\X\) and \(\sfd_\Y\).
\subsection{Reminder on direct and inverse limits}\label{ss:DL_IL}
Let us briefly recall the definitions of direct limit and
inverse limit in an arbitrary category. For a thorough account
on this topic, we refer e.g.\ to the classical reference
\cite{MacLane98}.
\medskip

Fix a directed set \((I,\leq)\) and an arbitrary category \(\mathscr C\).
By a \emph{direct system} in \(\mathscr C\) over \(I\) we mean a couple
\(\big(\{X_i\}_{i\in I},\{\varphi_{ij}\}_{i\leq j}\big)\),
where \(\{X_i\,:\,i\in I\}\) is a family of objects of \(\mathscr C\),
while \(\varphi_{ij}\colon X_i\to X_j\) is a morphism for every
\(i,j\in I\) with \(i\leq j\), such that the following properties hold:
\begin{itemize}
\item[\({\rm DS}_1)\)] \(\varphi_{ii}\) is the identity of \(X_i\)
for every \(i\in I\).
\item[\({\rm DS}_2)\)] \(\varphi_{ik}=\varphi_{jk}\circ\varphi_{ij}\)
holds for every \(i,j,k\in I\) with \(i\leq j\leq k\).
\end{itemize}
A given couple \(\big(X,\{\varphi_i\}_{i\in I}\big)\) -- where
\(X\) is an object of \(\mathscr C\) and \(\varphi_i\) is a
morphism \(\varphi_i\colon X_i\to X\) for every \(i\in I\) --
is said to be the \emph{direct limit} of
\(\big(\{X_i\}_{i\in I},\{\varphi_{ij}\}_{i\leq j}\big)\)
provided it holds that:
\begin{itemize}
\item[\({\rm DL}_1)\)] \(\big(X,\{\varphi_i\}_{i\in I}\big)\) is a
\emph{target} for the direct system
\(\big(\{X_i\}_{i\in I},\{\varphi_{ij}\}_{i\leq j}\big)\),
i.e., the diagram
\[\begin{tikzcd}
X_i \arrow[r,"\varphi_{ij}"] \arrow[rd,swap,"\varphi_i"] &
X_j \arrow[d,"\varphi_j"] \\
& X
\end{tikzcd}\]
is commutative for every \(i,j\in I\) with \(i\leq j\).
\item[\({\rm DL}_2)\)] Given any target \(\big(Y,\{\psi_i\}_{i\in I}\big)\) for
the direct system \(\big(\{X_i\}_{i\in I},\{\varphi_{ij}\}_{i\leq j}\big)\),
there exists a unique morphism \(\Phi\colon X\to Y\) such that the diagram
\[\begin{tikzcd}
X_i \arrow[r,"\varphi_i"] \arrow[rd,swap,"\psi_i"] &
X \arrow[d,"\Phi"] \\
& Y
\end{tikzcd}\]
commutes for every \(i\in I\).
\end{itemize}
The property stated in \({\rm DL}_2)\) is called the
\emph{universal property}.
An arbitrary direct system needs not admit a direct limit, but
whenever the direct limit exists, it is unique up to unique isomorphism. By abuse of notation, we sometimes refer to the object $ X $ as the direct limit of \(\big(\{X_i\}_{i\in I},\{\varphi_{ij}\}_{i\leq j}\big)\).
\smallskip

By an \emph{inverse system} in \(\mathscr C\) over \(I\) we mean
a couple \(\big(\{X_i\}_{i\in I},\{P_{ij}\}_{i\leq j}\big)\) --
where \(\{X_i\,:\,i\in I\}\) is a family of objects of \(\mathscr C\),
while \(P_{ij}\colon X_j\to X_i\) is a morphism for every \(i,j\in I\)
with \(i\leq j\) -- such that the following properties hold:
\begin{itemize}
\item[\({\rm IS}_1)\)] \(P_{ii}\) is the identity of \(X_i\) for every \(i\in I\).
\item[\({\rm IS}_2)\)] \(P_{ik}=P_{ij}\circ P_{jk}\) holds
for every \(i,j,k\in I\) with \(i\leq j\leq k\).
\end{itemize}
A couple \(\big(X,\{P_i\}_{i\in I}\big)\) --  where \(X\) is an
object of \(\mathscr C\) and \(P_i\colon X\to X_i\) is a morphism
for any \(i\in I\) -- is said to be the \emph{inverse limit} of
\(\big(\{X_i\}_{i\in I},\{P_{ij}\}_{i\leq j}\big)\) provided it holds that:
\begin{itemize}
\item[\({\rm IL}_1)\)] The diagram
\[\begin{tikzcd}
X \arrow[rd,"P_i"] \arrow[d,swap,"P_j"] & \\
X_j \arrow[r,swap,"P_{ij}"] & X_i
\end{tikzcd}\]
commutes for every \(i,j\in I\) with \(i\leq j\).
\item[\({\rm IL}_2)\)] Given any couple
\(\big(Y,\{Q_i\}_{i\in I}\big)\) satisfying the property in \({\rm IL}_1)\)
-- namely, \(Q_i=P_{ij}\circ Q_j\) for any \(i,j\in I\) with \(i\leq j\) --
there exists a unique morphism \(\Pi\colon Y\to X\) such that the diagram
\[\begin{tikzcd}
Y \arrow[r,"\Pi"] \arrow[rd,swap,"Q_i"] &
X \arrow[d,"P_i"] \\
& X_i
\end{tikzcd}\]
commutes for every \(i\in I\).
\end{itemize}
The property in \({\rm IL}_2)\) is referred to as
the \emph{universal property}. An inverse system
does not necessarily admit an inverse limit, but when the
inverse limit exists, it is unique up to unique isomorphism.
\medskip

Let us now spend a few words on direct and inverse limits in
the category of groups, which will play a central role in the
rest of this paper. We refer the reader to \cite{lang84} for
more on this topic.
\begin{theorem}[Direct/inverse limits of groups]
Direct limits and inverse limits always exist in the
category of groups.
\end{theorem}
For the usefulness of the reader, we also recall the explicit
description of direct and inverse limits of groups; a similar
construction works on many other algebraic structures, such as
rings, modules, or algebras. For the rest of this section,
let \((I,\leq)\) be a fixed directed set.
\medskip

Let \(\big(\{G_i\}_{i\in I},\{\varphi_{ij}\}_{i\leq j}\big)\)
be a direct system of groups. We define an equivalence
relation \(\sim\) on the set \(\bigsqcup_{i\in I}G_i\):
given any \(x\in G_i\) and \(y\in G_j\), we declare that
\(x\sim y\) provided there exists \(k\in I\) with \(i,j\leq k\)
such that \(\varphi_{ik}(x)=\varphi_{jk}(y)\). The equivalence
class of \(x\) is denoted by \([x]_\sim\).
Then we define the group \(G\) as
\begin{equation}\label{eq:def_DL_gps}
G\coloneqq\bigsqcup_{i\in I}G_i\Big/\sim,
\end{equation}
the group operation \(\cdot\colon G\times G\to G\) being
defined as follows: given \(x\in G_i\) and \(y\in G_j\), we set
\[
[x]_\sim\cdot[y]_\sim\coloneqq\big[\varphi_{ik}(x)\cdot\varphi_{jk}(y)
\big]_\sim\in G,\quad\text{ for some (thus any) }k\in I
\text{ with }i,j\leq k.
\]
It can be readily checked that this operation is well-defined
(i.e., it does not depend on the specific choice of the representatives
\(x,y\) and of \(k\)) and that the resulting structure \((G,\cdot)\)
is a group. Moreover, for any \(i\in I\) we define the map
\(\varphi_i\colon G_i\to G\) as \(\varphi_i(x)\coloneqq[x]_\sim\)
for all \(x\in G_i\). Then each map \(\varphi_i\) is a group homomorphism
and \(\big(G,\{\varphi_i\}_{i\in I}\big)\) is the direct
limit of \(\big(\{G_i\}_{i\in I},\{\varphi_{ij}\}_{i\leq j}\big)\) in
the category of groups. It is worth to point out that for any \(x\in G\)
there exist \(i\in I\) and \(x_i\in G_i\) such that \(x=\varphi_i(x_i)\).
\medskip

Now let \(\big(\{G_i\}_{i\in I},\{P_{ij}\}_{i\leq j}\big)\) be an inverse
system of groups. Then we define the group \(G\) as
\begin{equation}\label{eq:def_IL_groups}
G\coloneqq\Big\{x=(x_i)_{i\in I}\in\prod\nolimits_{i\in I}G_i\;\Big|
\;x_i=P_{ij}(x_j)\text{ for every }i,j\in I\text{ with }i\leq j\Big\},
\end{equation}
which is a subgroup of the direct product \(\prod_{i\in I}G_i\) (that
has a natural group structure with respect to the elementwise operation).
Moreover, for any \(i\in I\) we define \(P_i\colon G\to G_i\) as
\(P_i(x)\coloneqq x_i\) for every \(x=(x_i)_{i\in I}\in G\).
It holds that each map \(P_i\) is a group homomorphism and
\(\big(G,\{P_i\}_{i\in I}\big)\) is the inverse limit of
\(\big(\{G_i\}_{i\in I},\{P_{ij}\}_{i\leq j}\big)\) in the
category of groups.
\subsection{Scalable groups}
\label{ss:scalable_gps}
Let \(G\) be a group. Given any element \(x\in G\), we denote
by \(L_x\colon G\to G\) and \(R_x\colon G \to G\) the
\emph{left translation map} and the \emph{right translation map} at \(x\),
respectively. Namely,
\[\begin{split}
L_x(y)\coloneqq xy&\quad\text{ for every }y\in G,\\
R_x(y)\coloneqq yx&\quad\text{ for every }y\in G.
\end{split}\]
Moreover, we shall denote by \({\sf Op}\colon G\times G\to G\)
the \emph{group multiplication map} \((x,y)\mapsto xy\) and by
\({\sf Inv}\colon G\to G\) the \emph{inversion map} \(x\mapsto x^{-1}\).
We now define the notions of scalable group and topological scalable group. We stress that the definition of scalable group introduced in \cite{LDLM19} coincides with our definition of topological scalable group.
\begin{definition}[Scalable group]\label{def:scalable_gr}
Let \(G\) be a group. Then we say that
\(\{\delta_\lambda\}_{\lambda\in\R}\subseteq{\rm End}(G)\)
is a \emph{family of dilations} on \(G\) provided the
following properties hold:
\begin{itemize}
\item[\(\rm i)\)] \(\delta_0\) is given by
\(\delta_0(x)\coloneqq e\) for every \(x\in G\), where \(e\) stands
for the identity element of \(G\).
\item[\(\rm ii)\)]
\(\delta_\lambda\circ\delta_\mu=\delta_{\lambda\mu}\) for every
\(\lambda,\mu\in\R\).
\item[$ \rm iii) $] \(\delta_\lambda\in{\rm Aut}(G)\) for every \(\lambda\in\R\setminus\{0\}\).
\end{itemize}
We denote by \(\delta\colon\R\times G\to G\) the map
\((\lambda,x)\mapsto\delta_\lambda(x)\). The couple \((G,\delta)\) is called
a \emph{scalable group}. If $ G $ is a topological group and the dilation \(\delta\) is continuous, then \((G,\delta)\) is called a \emph{topological scalable group}.
\end{definition}
Subgroups and quotients of scalable groups are defined in the natural way. A map
\(\varphi\colon G\to G'\) between scalable groups is a \emph{morphism of scalable groups} provided it is a group homomorphism that satisfies
\[
\varphi\big(\delta_\lambda(x)\big)=\delta'_\lambda\big(\varphi(x)\big)
\quad\text{ for every }\lambda\in\R\text{ and }x\in G.
\]
With this notion of morphism at disposal, we can speak about the category of scalable groups. 
\begin{theorem}[Direct limits of scalable groups]\label{thm:DL_scalable_gps}
	Let $\big(\{G_i\}_{i\in I},\{\varphi_{ij}\}_{i\leq j}\big)$ be a direct system of scalable groups. Then the direct limit exists, and it is given by the direct limit \eqref{eq:def_DL_gps} in the category of groups
	together with a suitable dilation map.
\end{theorem}
\begin{proof}
Let \(\big(\{G_i\}_{i\in I},\{\varphi_{ij}\}_{i\leq j}\big)\)
be a direct system of scalable groups. In particular, it is a
direct system in the category of groups, thus call
\(\big(G,\{\varphi_i\}_{i\in I}\big)\) its direct limit (as a group).
We define a family of dilations \(\{\delta_\lambda\}_{\lambda\in\R}\)
on \(G\) as follows: given any \(\lambda\in\R\) and \(x\in G\),
there exist \(i\in I\) and \(x_i\in G_i\) such that \(\varphi_i(x_i)=x\);
then we set \(\delta_\lambda(x)\coloneqq\delta^i_\lambda(x_i)\),
where \(\delta^i\) stands for the dilation on \(G_i\). It can be readily
checked that \(\delta_\lambda\) is well-defined, that \((G,\delta)\) is
a scalable group, and that each map \(\varphi_i\) is a morphism of
scalable groups. To conclude, it only remains to prove the universal property:
let \(\big(H,\{\psi_i\}_{i\in I}\big)\) be any target of
\(\big(\{G_i\}_{i\in I},\{\varphi_{ij}\}_{i\leq j}\big)\) in the
category of scalable groups. In particular, it is a target in the category
of groups, thus there exists a unique group homomorphism
\(\Phi\colon G\to H\) such that \(\Phi\circ\varphi_i=\psi_i\)
for all \(i\in I\). Let us show that \(\Phi\) preserves the
dilation: if \(\lambda\in\R\) and \(x\in G\), then there exist
\(i\in I\) and \(x_i\in G_i\) such that \(\varphi_i(x_i)=x\), whence
\[
\Phi\big(\delta_\lambda(x)\big)=
(\Phi\circ\varphi_i)\big(\delta^i_\lambda(x_i)\big)=
\psi\big(\delta^i_\lambda(x)\big)=\delta'_\lambda\big(\psi_i(x_i)\big)
=\delta'_\lambda\big(\Phi(x)\big),
\]
where \(\delta'\) stands for the dilation on \(H\). This proves the
universal property, so that \(\big(G,\{\varphi_i\}_{i\in I}\big)\) is
the direct limit of \(\big(\{G_i\}_{i\in I},\{\varphi_{ij}\}_{i\leq j}\big)\)
in the category of scalable groups, as required.
\end{proof}

For the sake of completeness, we also give a proof of existence of inverse limits in the category of scalable groups.
\begin{theorem}[Inverse limits of scalable groups]\label{thm:IL_scalable_gps}
	Let \(\big(\{G_i\}_{i\in I},\{P_{ij}\}_{i\leq j}\big)\) be an inverse system
	of scalable groups. Then the inverse limit exists, and it is given by the inverse limit \eqref{eq:def_IL_groups} in the category of groups
	together with a suitable dilation map.
\end{theorem}
\begin{proof}
	Let \(\big(\{G_i\}_{i\in I},\{P_{ij}\}_{i\leq j}\big)\) be an inverse system
	of scalable groups. In particular, it is an inverse system in the category of
	groups, whose inverse limit we denote by \(\big(G,\{P_i\}_{i\in I}\big)\).
	We define the family of dilations \(\{\delta_\lambda\}_{\lambda\in\R}\) on \(G\)
	in the following way: given \(\lambda\in\R\) and \(x\in G\), we define
	\(\delta_\lambda(x)\) as the unique element of \(G\) such that
	\(P_i\big(\delta_\lambda(x)\big)=\delta^i_\lambda\big(P_i(x)\big)\)
	for some (thus any) \(i\in I\), where by \(\delta^i\) we mean the dilation on \(G_i\).
	It turns out that \((G,\delta)\) is a scalable group and each map \(P_i\) is
	a scalable group morphism. Let us now prove the universal property: fix a
	scalable group \(H\) and a family \(\{Q_i\}_{i\in I}\) of scalable group
	morphisms \(Q_i\colon H\to G_i\) such that \(Q_i=P_{ij}\circ Q_j\) for every
	\(i,j\in I\) with \(i\leq j\). Since each map \(Q_i\) is in particular a
	group homomorphism, there exists a unique group homomorphism
	\(\Phi\colon H\to G\) such that \(Q_i=P_i\circ\Phi\) for all \(i\in I\).
	Call \(\delta'\) the dilation on \(H\). If \(\lambda\in\R\) and \(y\in H\), then
	\[
	P_i\big(\delta_\lambda(\Phi(y))\big)=
	\delta^i_\lambda\big(P_i(\Phi(y))\big)=
	\delta^i_\lambda\big(Q_i(y)\big)=
	Q_i\big(\delta'_\lambda(y)\big)\quad\text{ for every }i\in I,
	\]
	thus accordingly \(\delta_\lambda\big(\Phi(y)\big)=\Phi\big(\delta'_\lambda(y)\big)\).
	This shows that \(\Phi\) is a morphism of scalable groups, whence the universal
	property is proven. Therefore, we have that \(\big(G,\{P_i\}_{i\in I}\big)\) is
	the inverse limit of \(\big(\{G_i\}_{i\in I},\{P_{ij}\}_{i\leq j}\big)\)
	in the category of scalable groups, as required.
\end{proof}
\section{Direct limits of (complete) metric scalable groups}
\label{s:DL_of_CMSG}
\subsection{Scalable groups with distances}
\label{ss:sc_gps_dist}
As a first step towards constructing direct limits of complete metric scalable groups, we study scalable groups equipped with a suitable pseudometric.
\begin{definition}[Compatible distance]\label{def:compatible_dist}
Let \((G,\delta)\) be a scalable group. Let \(\sfd\) be a pseudodistance on \(G\).
Then we say that \(\sfd\) is \emph{compatible} with \((G,\delta)\), if $ \sfd $ is left-invariant and
\[
\sfd\big(\delta_\lambda(x),\delta_\lambda(y)\big)=|\lambda|\,\sfd(x,y)
\quad\text{ for every }\lambda\in\R\text{ and }x,y\in G.
\]
\end{definition}
\noindent
Recall that a pseudodistance \(\sfd\) on a group \(G\) is said to be
\emph{left-invariant} provided it satisfies
\[
\sfd(xy,xz)=\sfd(y,z)\quad\text{ for every }x,y,z\in G.
\]
In the following lemma we collect some technical continuity results on scalable groups with compatible pseudodistances. These results will be needed in the construction of direct limits of (complete) metric scalable groups. Recall Definition \ref{def:modulus_cont} for the modulus of continuity.
\begin{lemma}
\label{lem:cont_operations}
Let \((G,\delta)\) be a scalable group endowed with a compatible
pseudodistance \(\sfd\). Then
for every $ x,y\in G $ and $ \eps >0 $, it holds that
\begin{enumerate}[label=\rm\alph*)]
	\item \label{it:Rx_unif_cont} \(\omega_{R_x}(e;\eps)=\omega_{R_x}(y;\eps)\)\,;
	\item \label{it:inv_vs_Rx} 
	\(\omega_{\sf Inv}(x;\eps)=\omega_{R_{x^{-1}}}(e;\eps)\)\,;
	\item \label{it:Rx_extension} $ \inf_{y\in B(x,\eps/3)} \omega_{R_y}(e;\eps) \geq \omega_{R_x}(e;\eps/3) $\,.
\end{enumerate}
Moreover,
\begin{enumerate}[label=\rm\alph*)]
	\setcounter{enumi}{3}
	\item \label{it:Op_Cauchy} if \(R_x\) is continuous at \(e\) for
	any \(x\in G\), then \({\sf Op}\colon G\times G\to G\) is
	Cauchy-continuous;
	\item \label{it:delta_Cauchy} if \(\delta(\cdot,x)\) is
	continuous for any \(x\in G\), then \(\delta\colon\R\times G\to G\)
	is Cauchy-continuous.
\end{enumerate}
\end{lemma}
\begin{proof}
\ \\
\ref{it:Rx_unif_cont} Fix any \(x,y,y'\in G\) and \(\eps>0\).
Choose any \(\eta<\omega_{R_x}(y';\eps)\). Then for every
\(z\in B_\eta(y)\) we have \(\sfd(y'y^{-1}z,y')=\sfd(z,y)<\eta\),
thus \(\sfd(zx,yx)=\sfd(y'y^{-1}zx,y'x)<\eps\) and accordingly
\(\omega_{R_x}(y;\eps)\geq\eta\). This yields
\(\omega_{R_x}(y;\eps)\geq\omega_{R_x}(y';\eps)\), whence \ref{it:Rx_unif_cont} follows
by the arbitrariness of \(x,y,y'\in G\) and \(\eps>0\).\\
\ref{it:inv_vs_Rx} Let \(x\in G\) and \(\eps>0\) be fixed. We first prove
that \(\omega_{\sf Inv}(x;\eps)\leq\omega_{R_{x^{-1}}}(e;\eps)\). Pick
\(\eta<\omega_{\sf Inv}(x;\eps)\). If \(y\in G\) satisfies
\(\sfd(y,e)<\eta\), then \(\sfd(xy^{-1},x)=\sfd(y^{-1},e)=\sfd(e,y)<\eta\)
and accordingly
\[
\sfd\big(R_{x^{-1}}(y),R_{x^{-1}}(e)\big)=\sfd(yx^{-1},x^{-1})
=\sfd\big((xy^{-1})^{-1},x^{-1}\big)=
\sfd\big({\sf Inv}(xy^{-1}),{\sf Inv}(x)\big)<\eps.
\]
Hence, we get \(\omega_{R_{x^{-1}}}(e;\eps)\geq\eta\).
By arbitrariness of \(\eta\), we conclude that
\(\omega_{\sf Inv}(x;\eps)\leq\omega_{R_{x^{-1}}}(e;\eps)\).

In order to prove the converse inequality, fix any
\(\eta<\omega_{R_{x^{-1}}}(e;\eps)\). If \(y\in G\) satisfies
\(\sfd(y,x)<\eta\), then \(\sfd(y^{-1}x,e)=\sfd(x,y)<\eta\). This implies that
\[
\sfd\big({\sf Inv}(y),{\sf Inv}(x)\big)=\sfd(y^{-1},x^{-1})
=\sfd(y^{-1}xx^{-1},x^{-1})=
\sfd\big(R_{x^{-1}}(y^{-1}x),R_{x^{-1}}(e)\big)<\eps,
\]
thus \(\omega_{\sf Inv}(x;\eps)\geq\eta\). By arbitrariness of \(\eta\),
we conclude that \(\omega_{R_{x^{-1}}}(e;\eps)\leq\omega_{\sf Inv}(x;\eps)\). \\
\ref{it:Rx_extension} Let us denote $ \eta \coloneqq \omega_{R_x}(e;\eps/3) $ and assume, without loss of generality,  that $ \eta > 0 $. It is then enough to observe that, given any $ z\in B_\eta(e) $ and $ y\in B_{\eps/3}(x) $, we have 
\[
\sfd\big(R_y(e),R_y(z)\big) =  \sfd(y,zy) \leq \sfd(y,x) + \sfd(x,zx) + \sfd(zx,zy) < \eps.
\]
\ref{it:Op_Cauchy} Suppose \(R_x\) is continuous at \(e\) for any \(x\in G\).
Let \(T\subseteq G\times G\) be a \(\sfd_{G\times G}\)-totally bounded set.
The projection \(T'\) of the set \(T\) on the second coordinate is
\(\sfd_G\)-totally bounded, as a consequence of the fact that the projection
map is Lipschitz. Now fix \(\eps>0\). Covering the set $ T' $ by a finite number of balls of radius $ \eps/3 $ we get by \ref{it:Rx_extension} and the continuity of $ R_y $ at every $ y\in G $ that
\[
\eta' \coloneqq \inf_{y\in T'}\omega_{R_y}(e;\eps) > 0.
\] 
Denote $ \eta \coloneqq \min\big\{\eps,\eta' \big\}
>0 $ and let us show that
\begin{equation}\label{eq:cont_operations_aux1}
\omega_{\sf Op}\big((x,y);2\eps\big)\geq \eta \quad\text{ for every }(x,y)\in T.
\end{equation}
To prove it, fix any \((x,y)\in T\). Given any
\((x',y')\in B_\eta(x,y)\), it holds that
\[
\sfd(xy,x'y')\leq \sfd(xy,x'y) + \sfd(x'y,x'y') < \eps + \eta \leq 2\eps.
\]
This shows \eqref{eq:cont_operations_aux1}. Thanks to
Lemma \ref{lemma:Cauchy_continuity_vs_moduli}, we thus conclude
that the group multiplication map \({\sf Op}\colon G\times G\to G\)
is Cauchy-continuous, as desired.\\
\ref{it:delta_Cauchy} Suppose \(\delta(\cdot,x)\) is continuous for any \(x\in G\).
Let \(T\subseteq\R\times G\) be totally bounded. Fix \(\eps>0\).
Denote by \(T_\R\subseteq\R\) and \(T_G\subseteq G\) the totally bounded projections of
\(T\) on the first and the second components, respectively. Call
\(M\coloneqq\sup_{\lambda\in T_\R}|\lambda|<+\infty\) and choose any
\(x_1,\ldots,x_n\in T_G\) for which
\(T_G\subseteq\bigcup_{i=1}^n B_{\eps/2M}(x_i)\). Let
\[
\eta(\lambda)\coloneqq
\frac{1}{2}\min\big\{\omega_{\delta(\cdot,x_1)}(\lambda;\eps),\ldots,
\omega_{\delta(\cdot,x_n)}(\lambda;\eps)\big\}>0
\quad\text{ for every }\lambda\in\R.
\]
Now pick \(\lambda_1,\ldots,\lambda_m\in T_\R\) such that
\(T_\R\subseteq\bigcup_{j=1}^m B_{\eta(\lambda_j)}(\lambda_j)\).
We claim that choosing
\[
\eta \coloneqq \min\big\{M,\eps/2M,\eta(\lambda_1),\ldots,\eta(\lambda_m)\big\}
\]
gives
\begin{equation}\label{eq:cont_operations_aux2}
\inf_{(\lambda,x)\in T}\omega_\delta\big((\lambda,x);5\eps\big)\geq \eta.
\end{equation}
To prove \eqref{eq:cont_operations_aux2}, fix \((\lambda,x)\in T\) and \((\lambda',x')\in\R\times G\)
such that \(\sfd_{\R\times G}\big((\lambda,x),(\lambda',x')\big)<\eta\).
Consider \(i\in \{1,\ldots,n\}\) and \(j\in \{1,\ldots,m\}\) for which
\(|\lambda-\lambda_j|<\eta(\lambda_j)\) and \(\sfd(x,x_i)<\eps/2M\).
Observe that
\[\begin{split}
\sfd\big(\delta(\lambda,x),\delta(\lambda',x')\big)\leq&\;
\sfd\big(\delta_\lambda(x),\delta_\lambda(x_i)\big)+
\sfd\big(\delta_\lambda(x_i),\delta_{\lambda_j}(x_i)\big)+
\sfd\big(\delta_{\lambda_j}(x_i),\delta_{\lambda'}(x_i)\big)\\
&+\sfd\big(\delta_{\lambda'}(x_i),\delta_{\lambda'}(x)\big)+
\sfd\big(\delta_{\lambda'}(x),\delta_{\lambda'}(x')\big).
\end{split}\]
Let us estimate the five terms appearing in the right-hand side of the
previous formula. First,
\[\begin{split}
\sfd\big(\delta_\lambda(x),\delta_\lambda(x_i)\big)&=
|\lambda|\,\sfd(x,x_i)\leq M\frac{\eps}{2M}<\eps,\\
\sfd\big(\delta_{\lambda'}(x_i),\delta_{\lambda'}(x)\big)&=
|\lambda'|\,\sfd(x_i,x)\leq\big(\eta+|\lambda|\big)\,\sfd(x_i,x)<
2M\frac{\eps}{2M}=\eps,\\
\sfd\big(\delta_{\lambda'}(x),\delta_{\lambda'}(x')\big)&=
|\lambda'|\,\sfd(x,x')\leq\big(\eta+|\lambda|\big)\eta
\leq 2M\frac{\eps}{2M}=\eps.
\end{split}\]
Moreover, since \(|\lambda-\lambda_j|<\eta(\lambda_j)
<\omega_{\delta(\cdot,x_i)}(\lambda_j;\eps)\), we have that
\(\sfd\big(\delta_\lambda(x_i),\delta_{\lambda_j}(x_i)\big)<\eps\).
Finally, since \(|\lambda_j-\lambda'|\leq|\lambda_j-\lambda|+\eta<
2\,\eta(\lambda_j)\leq\omega_{\delta(\cdot,x_i)}(\lambda_j;\eps)\), we see
that \(\sfd\big(\delta_{\lambda_j}(x_i),\delta_{\lambda'}(x_i)\big)<\eps\).
All in all, we have proven that
\(\sfd\big(\delta(\lambda,x),\delta(\lambda',x')\big)<5\eps\),
thus obtaining the claim \eqref{eq:cont_operations_aux2}. By taking
Lemma \ref{lemma:Cauchy_continuity_vs_moduli} into account, we can conclude
that \(\delta\colon\R\times G\to G\) is Cauchy-continuous, as desired.
\end{proof}
\begin{remark}{\rm
Observe that items \ref{it:Rx_unif_cont} -- \ref{it:Op_Cauchy} hold also without existence of the dilation \(\delta\).\fr}\end{remark}
\begin{remark}{\rm Note that, given a scalable group $ (G,\delta) $ equipped with a compatible distance $ \sfd $, the continuity of the map \(\delta(\cdot,x)\colon\R\to G\) for	\(x\in G\setminus\{e\}\) is not automatic even when $G$ is an abelian topological group and $ \sfd $ is geodesic. Indeed, let
\(f\colon\big(\R\setminus\{0\},\cdot\big)\to(\R,+)\) be a non-continuous
group homomorphism and consider the abelian group \(G =(\R^2,+)\). Let \(\mathcal L_\theta\) denote the rotation in \(\R^2\) by an angle \(\theta\in\R\). Define the
maps \(\delta_\lambda\colon\R^2\to\R^2\) as
\[
\delta_\lambda=(\lambda\,{\rm id}_{\R^2})\circ\mathcal L_{f(\lambda)}
\quad\text{ for every }\lambda\in\R\setminus\{0\},\quad \text{and }\delta_0 \equiv 0,
\]
where \({\rm id}_{\R^2}\) stands for the identity map.
Clearly, \(\lambda\mapsto\delta_\lambda(x)\) is not continuous
for any \(x\in\R^2\setminus\{0\}\), but we claim that \((\R^2,\delta)\)
is a scalable group and that \(\delta\) is a metric dilation with respect
to the Euclidean distance. It suffices to notice that \(\delta_\lambda\)
is a linear bijection for each \(\lambda\in\R\setminus\{0\}\) satisfying
\[
\delta_s\circ \delta_t=(s\,{\rm id}_{\R^2})\circ\mathcal L_{f(s)}\circ
(t\,{\rm id}_{\R^2})\circ\mathcal L_{f(t)}=(st\,{\rm id}_{\R^2})\circ
\mathcal L_{f(s)+f(t)}=(st\,{\rm id}_{\R^2})\circ\mathcal L_{f(st)}=
\delta_{st}
\]
and, since \(\mathcal L_\theta\) is an isometry,
\(\big|\delta_\lambda(x)-\delta_\lambda(y)\big|=\big|(\lambda\,{\rm id}_{\R^2})
\circ\mathcal L_{f(\lambda)}(x-y)\big|=|\lambda||x-y|\).
\fr}\end{remark}
The well-known fact that a linear map between normed spaces is continuous
(at the origin) if and only if it is Lipschitz, can be generalized to the
setting of scalable groups:
\begin{proposition}\label{prop:morph_cont_vs_Lip}
	Let \((G,\delta)\) and \((G',\delta')\) be scalable groups. Let
	\(\sfd\) (resp.\ \(\sfd'\)) be a compatible pseudodistance on \(G\)
	(resp.\ on \(G'\)). Let \(\varphi\colon G\to G'\) be a morphism
	of scalable groups. Then \(\varphi\) is continuous at \(e\in G\) if and
	only if it is Lipschitz.
\end{proposition}
\begin{proof}
	Sufficiency is obvious. To prove necessity, suppose \(\varphi\) is
	continuous at \(e\). Denote by \(e'\) the identity element of
	\(G'\). We claim that there exists a constant \(C>0\) such that
	\begin{equation}\label{eq:suff_cond_Lip_claim}
	\sfd'\big(\varphi(x),e'\big)\leq C\,\sfd(x,e)\quad\text{ for every }x\in G.
	\end{equation}
	In order to prove it, we argue by contradiction: suppose there exists a
	sequence \((x_n)_n\subseteq G\) such that
	\(\sfd'\big(\varphi(x_n),e'\big)>n\,\sfd(x_n,e)\) for all \(n\in\N\).
	Call \(\lambda_n\coloneqq 1/\big(n\,\sfd(x_n,e)\big)>0\)
	and \(y_n\coloneqq\delta_{\lambda_n}(x_n)\in G\) for every \(n\in\N\). Then
	\(\sfd(y_n,e)=\lambda_n\,\sfd(x_n,e)=1/n\to 0\) as \(n\to\infty\), while
	\[
	\sfd'\big(\varphi(y_n),e'\big)=
	\sfd'\big(\delta'_{\lambda_n}\big(\varphi(x_n)\big),e'\big)
	=\lambda_n\,\sfd'\big(\varphi(x_n),e'\big)>1\quad\text{ for every }n\in\N.
	\]
	Since \(\varphi(e)=e'\), this contradicts the continuity of \(\varphi\)
	at \(e\). Therefore, the claim \eqref{eq:suff_cond_Lip_claim} is proven.
	By exploiting \eqref{eq:suff_cond_Lip_claim} and the left-invariance
	of \(\sfd\) and \(\sfd'\), we finally conclude that
	\[
	\sfd'\big(\varphi(x),\varphi(y)\big)=\sfd'\big(\varphi(y^{-1}x),e'\big)
	\leq C\,\sfd(y^{-1}x,e)=C\,\sfd(x,y)\quad\text{ for every }x,y\in G,
	\]
	whence the map \(\varphi\) is Lipschitz, as required.
\end{proof}
\subsection{Metric scalable groups}
\label{ss:DL_MSG}
The following definition is introduced in \cite{LDLM19}. Recall the definition of topological scalable group (Defintion \ref{def:scalable_gr}) and of compatible distance (Definition \ref{def:compatible_dist}).
\begin{definition}[Metric scalable group]\label{def:metric_scalable_gp}
A \emph{metric scalable group}  (or briefly \emph{MSG}) is a triple \((G,\delta,\sfd)\),
where \((G,\delta)\) is a topological scalable group, while \(\sfd\) is a compatible distance on \((G,\delta)\) that induces the topology of \(G\).
A \emph{complete metric scalable group} is called \emph{CMSG}.
\end{definition}
Let \((G,\delta,\sfd)\), \((G',\delta',\sfd')\) be metric scalable groups.
Then a map \(\varphi\colon G\to G'\) is said to be a
\emph{morphism of metric scalable groups} provided it is a \(1\)-Lipschitz
morphism of scalable groups. With this notion at our disposal, we
can speak about the category of MSGs. We are now going to construct the metric limits of MSGs, as discussed in the Introduction.
\medskip

Given a direct system \(\big(\{G_i\}_{i\in I},\{\varphi_{ij}\}_{i\leq j}\big)\) of metric scalable groups, we denote by \(\delta^i\) and
\(\sfd_i\) the dilation and the distance on \(G_i\), respectively.
Since \(\big(\{G_i\}_{i\in I},\{\varphi_{ij}\}_{i\leq j}\big)\) is -- a fortiori -- a direct system also in the category of scalable groups, it admits a direct limit \(\big(G',\{\varphi'_i\}_{i\in I}\big)\)
in such category (recall Theorem \ref{thm:DL_scalable_gps}). We define the \emph{infimum-pseudodistance} $ \sfd' $ on $ G' $ in the following way:
\begin{equation}\label{eq:inf-pseudodist}
\sfd'(x,y)\coloneqq\inf\Big\{\sfd_i(x_i,y_i)\;\Big|\;
i\in I,\,x_i,y_i\in G_i,\,\varphi'_i(x_i)=x,\,\varphi'_i(y_i)=y\Big\}
\quad\text{ for every }x,y\in G'.
\end{equation}
One can easily check that $ \sfd' $ is compatible with the scalable group \((G',\delta')\). Setting
\begin{equation*}\label{eq:inf-pseudodist-zeroset}
Z(\sfd')\coloneqq\big\{x\in G'\;\big|\;\sfd'(x,e')=0\big\},
\end{equation*}
one may consider the induced quotient metric space $ \big(G'/Z(\sfd'),\sfd \big) $, where $ \sfd $ is defined as
\begin{equation*}\label{eq:def_d}
\sfd\big(xZ(\sfd'),yZ(\sfd')\big)\coloneqq\sfd'(x,y)
\quad\text{ for every }x,y\in G'.
\end{equation*}
Moreover, we may define the map
\(\delta\colon\R\times\big(G'/Z(\sfd')\big)\to G'/Z(\sfd')\) as
\begin{equation}\label{eq:def_delta}
	\delta\big(t,xZ(\sfd')\big)=\delta_t\big(xZ(\sfd')\big)\coloneqq
	\delta'_t(z)Z(\sfd')\quad\text{ for every }t\in\R\text{ and }x\in G'.
\end{equation}
In the proof of Lemma \ref{lem:metric_limit_is_msg} we show that $ Z(\sfd')$ is a scalable subgroup of $ G' $. Hence $ \delta $ is well-posed, i.e., it does not depend on the choice of representative.
\begin{definition}[Metric limit]
	\label{def:metric_limit} Let \(\big(\{G_i\}_{i\in I},\{\varphi_{ij}\}_{i\leq j}\big)\) be a direct system of metric scalable groups. Let $ (G',\delta') $ be its direct limit in the category of scalable groups and let $ \sfd' $ be the infimum-pseudodistance on $ G' $.
	 The triple \(\big(G'/Z(\sfd'),\delta,\sfd \big)\) is called the \emph{metric limit} of \(\big(\{G_i\}_{i\in I},\{\varphi_{ij}\}_{i\leq j}\big)\).
\end{definition}
We stress that, in general, the metric limit \(\big(G'/Z(\sfd'),\delta,\sfd \big)\) is not a metric scalable group. Indeed, the scalable subgroup $ Z(\sfd') $ might fail to be normal, so that \(G'/Z(\sfd')\) does not have a
natural group structure. Moreover, even when \(Z(\sfd')\) is normal, the group operations on $ G'/Z(\sfd') $ might fail to be continuous with respect to $ \sfd' $.
We now impose conditions on the direct system  \(\big(\{G_i\}_{i\in I},\{\varphi_{ij}\}_{i\leq j}\big)\) that will be necessary and sufficient for our constructions. In particular, the conditions will guarantee that the metric limit is a metric scalable group.
\begin{definition}[Non-degenerate direct system of MSGs]
\label{def:non-deg_DS_MSG}
Let \(\big(\{G_i\}_{i\in I},\{\varphi_{ij}\}_{i\leq j}\big)\) be a direct system of metric scalable groups. Equip its direct limit \(\big(G',\{\varphi'_i\}_{i\in I}\big)\) in the category of scalable groups with the infimum-pseudodistance $ \sfd' $. Then we say that \(\big(\{G_i\}_{i\in I},\{\varphi_{ij}\}_{i\leq j}\big)\) is \emph{non-degenerate} provided the
following conditions are satisfied:
\begin{enumerate}[label={\rm (C{\arabic*})}]
	\item \label{eq:delta_non-deg_lem}$ \delta'(\cdot, x) \colon \R \to G' $ is continuous for all $ x\in G' $,
	\item \label{eq:Rx_non-deg_lem} $ R_x\colon G' \to G' $ is continuous at $ e $ for all $ x\in G' $.
\end{enumerate}
\end{definition}
As we are going to show in the following result, the continuity conditions
\ref{eq:delta_non-deg_lem} and \ref{eq:Rx_non-deg_lem} can be directly checked along the approximating
net \(\{G_i\}_{i\in I}\).
\begin{proposition}\label{prop:char_non-def_DS}
	Let \(\big(\{G_i\}_{i\in I},\{\varphi_{ij}\}_{i\leq j}\big)\)
	be a direct system of metric scalable groups. Denote by
	\(\big(G',\{\varphi'_i\}_{i\in I}\big)\) the direct limit of
	\(\big(\{G_i\}_{i\in I},\{\varphi_{ij}\}_{i\leq j}\big)\)
	as scalable groups and by \(\sfd'\) the infimum-pseudodistance on \(G'\).
	Then the conditions \ref{eq:delta_non-deg_lem} and \ref{eq:Rx_non-deg_lem} can be equivalently formulated as:
\begin{enumerate}[label={\rm (C\arabic*')}]
	\item \label{eq:delta_non-deg} For all  $ t\in \R $ and $ x\in G' $,
	\[\inf_{\eta>0}\;\sup_{\substack{s\in\R:\\|s-t|<\eta}}\;
	\inf\Big\{\sfd_i\big(\delta^i(s,x_i),\delta^i(t,y_i)\big)
	\;\Big|\;x_i,y_i\in(\varphi'_i)^{-1}(x)\Big\}=0;\]
	\item \label{eq:Rx_non-deg} For all $ x\in G' $,
	\[\inf_{\eta>0}\;\sup_{\substack{y_i\in G_i:\\\sfd_i(e_i,y_i)<\eta}}\;	\inf\Big\{\sfd_j(x_j,y_jx_j)\;\Big|\;
	x_j\in(\varphi'_j)^{-1}(x),\,
	y_j\in(\varphi'_j)^{-1}\big(\varphi'_i(y_i)\big)\Big\}=0.\]
\end{enumerate}
\end{proposition}
\begin{proof}
We just prove the equivalence \(\ref{eq:delta_non-deg_lem}\Leftrightarrow
\ref{eq:delta_non-deg}\), the argument for \(\ref{eq:Rx_non-deg_lem}\Leftrightarrow\ref{eq:Rx_non-deg}\) being similar.\\
	{\color{blue}\(\ref{eq:delta_non-deg_lem}\Longrightarrow\ref{eq:delta_non-deg}\)}
	Suppose \ref{eq:delta_non-deg_lem} holds. Fix \(x\in G'\) and \(t\in\R\).
	Given \(\eps>0\), there exists $ \eta > 0 $ such that
	\(\sfd'\big(\delta'_s(x),\delta'_t(x)\big)<\eps'\) for all
	\(s\in\R\) with \(|s-t|<\eta\). By definition of $ \sfd' $, there exist
	\(i\in I\) and points \(x_i,y_i\in G_i\) such that
	\(\varphi'_i(x_i)=\varphi'_i(y_i)=x\) and
	\(\sfd_i\big(\delta^i_s(x_i),\delta^i_t(y_i)\big)<\eps\).
	This implies \ref{eq:delta_non-deg}.\\
	{\color{blue}\ref{eq:delta_non-deg}$ \Longrightarrow $\ref{eq:delta_non-deg_lem}}
	Suppose \ref{eq:delta_non-deg} holds. Fix \(\eps>0\),
	\(x\in G'\), and \(t\in\R\). There exists \(\eta>0\)
	such that for every \(s\in\R\) with \(|s-t|<\eta\) there
	exist \(i\in I\), \(x_i,y_i\in(\varphi'_i)^{-1}(x)\) satisfying
	\(\sfd_i\big(\delta^i_s(x_i),\delta^i_t(y_i)\big)<\eps\).
	This implies \(\sfd'\big(\delta'_s(x),\delta'_t(x)\big)<\eps\),
	which shows that $ \delta'(\cdot,x) $ is continuous at $ t $. Thus \ref{eq:delta_non-deg_lem} holds.
\end{proof}
\begin{lemma}\label{lem:metric_limit_is_msg}
	Let \(\big(\{G_i\}_{i\in I},\{\varphi_{ij}\}_{i\leq j}\big)\)
	be a direct system of metric scalable groups satisfying \ref{eq:Rx_non-deg_lem}. Then its metric limit is a scalable group endowed with a compatible distance $ \sfd $. If also \ref{eq:delta_non-deg_lem} is satisfied, then the metric limit is a metric scalable group.
\end{lemma}
\begin{proof}
	We start by showing that $ Z(\sfd') $ is a normal scalable subgroup of $ G' $, which ensures that the quotient space $ \big(G'/Z(\sfd'),\delta \big) $ is a scalable group. Given any \(x,y\in Z(\sfd')\), we have that
	\[
	\sfd'(xy,e')\leq\sfd'(xy,x)+\sfd'(x,e')=\sfd'(y,e')+\sfd'(x,e')=0,
	\]
	which shows that \(xy\in Z(\sfd')\) as well. Moreover, we have
	\(\sfd'(x^{-1},e')=\sfd'(xx^{-1},x)=\sfd'(e',x)=0\), which implies
	that \(x^{-1}\in Z(\sfd')\). All in all, \(Z(\sfd')\) is a subgroup of \(G'\).
	Given \(\lambda\in\R\) and \(x\in Z(\sfd')\), it holds that
	\(\sfd'\big(\delta'_\lambda(x),e'\big)=
	\sfd'\big(\delta'_\lambda(x),\delta'_\lambda(e')\big)=
	|\lambda|\,\sfd'(x,e')=0\), so that \(\delta'_\lambda(x)\in Z(\sfd')\).
	This shows that \(Z(\sfd')\) is a scalable subgroup of \(G'\).

	To show that $ Z(\sfd') $ is normal, fix any \(x\in Z(\sfd')\) and \(y\in G\). Since the map \(R_y\) continuous at \(e'\) by \ref{eq:Rx_non-deg_lem}, we have that \(\omega_{R_y}(e';1/n)>0\) holds for every \(n\in\N\).
	Since \(\sfd'(x,e')=0<\omega_{R_y}(e';1/n)\), we deduce that
	\(\sfd'(y^{-1}xy,e')=\sfd'(xy,y)<1/n\). By letting \(n\to\infty\),
	we conclude that \(y^{-1}xy\in Z(\sfd')\), thus proving that
	\(Z(\sfd')\) is a normal subgroup of \(G'\), as required. Hence \((G'/Z(\sfd'), \delta)\) is a scalable group. Finally, it can be readily checked that the induced distance \(\sfd\) on \(G'/Z(\sfd')\) is compatible with \((G'/Z(\sfd'), \delta)\), which concludes the first part of the claim.
	
	Assume then that \ref{eq:delta_non-deg_lem} holds and let us prove that \((G'/Z(\sfd'), \delta, \sfd)\) is a metric scalable group. We need to verify that the dilation $ \delta \colon \R \times G \to G $, the group operation  $ {\sf Op}\colon G\times G \to G $ and the inversion map $ {\sf Inv}\colon G \to G $ are continuous. Observe that a morphism fixing the subgroup $ Z(\sfd') $ on $(G',\sfd') $ is continuous if and only if the induced map on  $ \big(G'/Z(\sfd'),\sfd\big) $ is continuous. Then by \ref{eq:delta_non-deg_lem} and \ref{eq:Rx_non-deg_lem}, the continuity of $ \delta $ and $ {\sf Op} $ follows from Lemma \ref{lem:cont_operations} \ref{it:delta_Cauchy} and Lemma \ref{lem:cont_operations} \ref{it:Op_Cauchy}, respectively. Moreover, the continuity of $ {\sf Inv} $ is given by \ref{eq:Rx_non-deg_lem} and Lemma \ref{lem:cont_operations} \ref{it:inv_vs_Rx}. Hence the metric limit \((G'/Z(\sfd'), \delta, \sfd)\) is a metric scalable group, as claimed.
\end{proof}
We prove now the part of Theorem \ref{thm:DL_equiv_non-deg} concerning metric scalable groups.
\begin{theorem}[Direct limits of MSGs]\label{thm:MSG_DL}
	Let  \(\big(\{G_i\}_{i\in I},\{\varphi_{ij}\}_{i\leq j}\big)\)
	be a direct system of metric scalable groups. Then the following are equivalent:
	\begin{itemize}
		\item[\(\rm i)\)] The direct system
		\(\big(\{G_i\}_{i\in I},\{\varphi_{ij}\}_{i\leq j}\big)\) is non-degenerate (in the sense of Definition \ref{def:non-deg_DS_MSG}).
		\item[\(\rm ii)\)] The direct limit of \(\big(\{G_i\}_{i\in I},\{\varphi_{ij}\}_{i\leq j}\big)\) in the category
		of metric scalable groups exists. Moreover, it equals the metric limit of \(\big(\{G_i\}_{i\in I},\{\varphi_{ij}\}_{i\leq j}\big)\) as in Definition \ref{def:metric_limit}.
	\end{itemize}
\end{theorem}
\begin{proof}
	{\color{blue}\({\rm i)}\Longrightarrow{\rm ii)}\)} Suppose i) holds. Let \(\big(G',\{\varphi'_i\}_{i\in I}\big)\) be the direct limit of \(\big(\{G_i\}_{i\in I},\{\varphi_{ij}\}_{i\leq j}\big)\) in the category of scalable groups. Denoting $ G\coloneqq G'/Z(\sfd') $, let \((G,\delta, \sfd \big)\) be the metric limit of \(\big(\{G_i\}_{i\in I},\{\varphi_{ij}\}_{i\leq j}\big)\).  By Lemma \ref{lem:metric_limit_is_msg}, \((G,\delta, \sfd \big)\) is a metric scalable group.	Let \(\pi\colon G'\to G\) be the projection \( \pi(x)=xZ(\sfd')\) and define $ \varphi_i \colon G_i \to G $ for each $ i\in I $ as
	\[
	\varphi_i\coloneqq\pi \circ \varphi'_i.
	\]
	Then by definition of the infimum-distance $ \sfd $, each $ \varphi_i \colon G_i \to G'$ is a 1-Lipschitz map, hence a morphism of metric scalable groups. Since \(\varphi_i=\varphi_j\circ\varphi_{ij}\) for all \(i,j\in I\),
	\(i\leq j\), \(\big(G,\{\varphi_i\}_{i\in I}\big)\) is a target of \(\big(\{G_i\}_{i\in I},\{\varphi_{ij}\}_{i\leq j}\big)\) in the
	category of metric scalable groups.
	
	To prove that it satisfies the universal property, fix any target
	\(\big(H,\{\psi_i\}_{i\in I}\big)\). Since \(\big(H,\{\psi_i\}_{i\in I}\big)\)
	is a target also in the category of scalable groups, there exists a
	unique morphism of scalable groups \(\Phi'\colon G'\to H\) such that
	\(\Phi'\circ\varphi'_i=\psi_i\) holds for every \(i\in I\). We claim then that the map $ \Phi' $ is 1-Lipschitz. Indeed, given \(i\in I\), \(x_i\in G_i\), and \(y_i\in G_i\) with \(\varphi'_i(x_i)=x\) and \(\varphi'_i(y_i)=y\), we have
	\[
	\sfd_H\big(\Phi'(x),\Phi'(y)\big)
	=\sfd_H\big((\Phi'\circ\varphi'_i)(x),(\Phi'\circ\varphi'_i)(y)\big)
	=\sfd_H\big(\psi_i(x_i),\psi_i(y_i)\big)\leq\sfd_i(x_i,y_i),
	\]
	where we denote by \(\sfd_H\) the distance on \(H\). This proves $ \sfd_H\big(\Phi'(x),\Phi'(y)\big) < \sfd'(x,y) $, as wanted.
	
	It now follows that $ \ker(\pi) = Z(\sfd') \sus \ker(\Phi') $, and so by the universal property of quotients, there exists a unique map \(\Phi\colon G\to H\) such that \(\Phi\circ \pi=\Phi'\). Moreover, for any \(xZ(\sfd'),yZ(\sfd')\in G\) it holds that
	\[
	\sfd_H\big(\Phi(xZ(\sfd')),\Phi(yZ(\sfd'))\big)=\sfd_H\big(\Phi'(x),\Phi'(y)\big)
	\leq\sfd'(x,y)=\sfd\big(xZ(\sfd'),yZ(\sfd')\big),
	\]
	which shows that also \(\Phi\) is a \(1\)-Lipschitz map. Hence \(\Phi\) is a morphism of metric scalable groups, and  \(\big(G,\{\varphi_i\}_{i\in I}\big)\) satisfies the universal property. Then the metric limit \((G,\delta,\sfd)\) is the direct limit of \(\big(\{G_i\}_{i\in I},\{\varphi_{ij}\}_{i\leq j}\big)\) in the
	category of metric scalable groups, yielding the sought conclusion.\\
{\color{blue}\({\rm ii)}\Longrightarrow{\rm i)}\)} Suppose ii) holds. Then, in particular, the metric limit of \(\big(\{G_i\}_{i\in I},\{\varphi_{ij}\}_{i\leq j}\big)\) is a metric scalable group. Hence the dilation map and the right translation are continuous on $ G'/Z(\sfd') $, which implies the continuity of these operations on $ G' $. Thus \ref{eq:delta_non-deg_lem} and \ref{eq:Rx_non-deg_lem} are satisfied, and $\big(\{G_i\}_{i\in I},\{\varphi_{ij}\}_{i\leq j}\big)$ is non-degenerate.
\end{proof}
\subsection{Complete metric scalable groups}\label{ss:DL_CMSG}
We consider the category of CMSGs as a subcategory of metric scalable groups; any CMSG is in particular a MSG, and the morphisms of CMSGs are the morphisms of MSGs.

Let \((G,\delta,\sfd)\) and \((\bar G,\bar\delta,\bar\sfd)\) be
a MSG and a CMSG, respectively. Suppose \((\bar G,\bar\sfd)\)
is the metric completion of \((G,\sfd)\), with isometric embedding
\(\iota\colon G\hookrightarrow\bar G\). If \(\iota\) is a morphism
of scalable groups, then we say that \(\bar G\)
is the \emph{completion of \(G\) as a metric scalable group}.
\begin{lemma}[Extension of MSGs]\label{lem:suff_cond_Cauchy_op}
Let \((G,\delta,\sfd)\) be a metric scalable group. Suppose
the inversion map \({\sf Inv}\colon G\to G\) is Cauchy-continuous.
Then the metric completion \((\bar G,\bar\sfd)\) of \((G,\sfd)\) can be
uniquely endowed with a metric scalable group structure
\((\bar G,\bar\delta,\bar\sfd)\) in such a way that
\(\bar G\) is the completion of \(G\) as a metric scalable group.
\end{lemma}
\begin{proof}
First of all, observe that the assumption on \(\sf Inv\) and
items b), d) of Lemma \ref{lem:cont_operations} grant that
\[\begin{split}
G\times G\ni(x,y)&\longmapsto\iota(xy)\in\bar G,\\
\R\times G\ni(\lambda,x)&\longmapsto\iota\big(\delta_\lambda(x)\big)\in\bar G,\\
G\ni x&\longmapsto\iota(x^{-1})\in\bar G
\end{split}\]
are Cauchy-continuous, where \(\iota\colon G\hookrightarrow\bar G\)
is the natural isometric embedding. Therefore, we know from Theorem
\ref{thm:extension_to_completion} that the above maps can be uniquely extended
to continuous operations
\[\begin{split}
\bar G\times\bar G\ni(x,y)&\longmapsto xy\in\bar G,\\
\R\times\bar G\ni(\lambda,x)&\longmapsto\bar\delta_\lambda(x)\in\bar G,\\
\bar G\ni x&\longmapsto x^{-1}\in\bar G.
\end{split}\]
By a standard approximation argument, one can readily check that the resulting
structure \((\bar G,\bar\delta,\bar\sfd)\) is a complete metric scalable group
and that \(\iota\) is a morphism of metric scalable groups.
\end{proof}
We now introduce a notion of non-degenerate direct system in the category
of CMSGs.
\begin{definition}[Non-degenerate direct system of CMSGs]
\label{def:non-deg_DS_CMSG}
Let \(\big(\{G_i\}_{i\in I},\{\varphi_{ij}\}_{i\leq j}\big)\) be a direct system of complete metric scalable groups. Equip its direct limit \(\big(G',\{\varphi'_i\}_{i\in I}\big)\) in the category of scalable groups with the infimum-pseudodistance $ \sfd' $. Then we say that \(\big(\{G_i\}_{i\in I},\{\varphi_{ij}\}_{i\leq j}\big)\) is \emph{non-degenerate} provided it
satisfies \ref{eq:delta_non-deg_lem} from Definition \ref{def:non-deg_DS_MSG}
and the following condition:
\begin{enumerate}[label={\rm (C3)}]
	\item \label{eq:inv_non-deg_lem}$ {\sf Inv} \colon G' \to G'$ is Cauchy-continuous.
\end{enumerate}
\end{definition}
Observe that, by Lemma \ref{lem:cont_operations} \ref{it:inv_vs_Rx}, property \ref{eq:inv_non-deg_lem} implies \ref{eq:Rx_non-deg_lem}. Hence, a non-degenerate direct system in the category of CMSGs is non-degenerate also as a direct system of MSGs. However, we do not
know whether the converse implication holds (but cf.\ Lemma \ref{lem:nilpotent_inv_cauchy}).
Notice also that, thanks to Lemma \ref{lem:suff_cond_Cauchy_op},
\ref{eq:inv_non-deg_lem} grants that the metric completion of the metric limit
of \(\big(\{G_i\}_{i\in I},\{\varphi_{ij}\}_{i\leq j}\big)\) is well-defined.
\medskip

In analogy with Proposition \ref{prop:char_non-def_DS}, condition
\ref{eq:inv_non-deg_lem} can be characterized in terms of the
net \(\{G_i\}_{i\in I}\). We omit the proof of the
following statement, as it is similar to the one of Proposition \ref{prop:char_non-def_DS}.
\begin{proposition}\label{prop:char_non-def_DS_complete}
	Let \(\big(\{G_i\}_{i\in I},\{\varphi_{ij}\}_{i\leq j}\big)\)
	be a direct system of complete metric scalable groups. Denote by
	\(\big(G',\{\varphi'_i\}_{i\in I}\big)\) the direct limit of
	\(\big(\{G_i\}_{i\in I},\{\varphi_{ij}\}_{i\leq j}\big)\)
	as scalable groups and by \(\sfd'\) the infimum-pseudodistance on \(G'\).
	Then the condition \ref{eq:inv_non-deg_lem} can be equivalently formulated as:
\begin{enumerate}[label={\rm (C3')}]

	\item \label{eq:inv_non-deg} For all $ T\subset G' $ totally bounded,
	\[
	\inf_{\eta>0}\;\sup_{\substack{x_i,y_i\in G_i:\\
			\varphi'_i(x_i)\in T,\\\sfd_i(y_i,x_i)<\eta}}\;
	\inf\Big\{\sfd_j(y_j^{-1},x_j^{-1})\;\Big|\;
	x_j\in(\varphi'_j)^{-1}\big(\varphi'_i(x_i)\big),\,
	y_j\in(\varphi'_j)^{-1}\big(\varphi'_i(y_i)\big)\Big\}=0.
	\]
\end{enumerate}
\end{proposition}
Together with Theorem \ref{thm:MSG_DL}, the following result proves Theorem \ref{thm:DL_equiv_non-deg}.
\begin{theorem}[Direct limits of CMSGs]
\label{thm:DL_of_CMSG}
Let \(\big(\{G_i\}_{i\in I},\{\varphi_{ij}\}_{i\leq j}\big)\) be a direct
system of complete metric scalable groups.
Then the following conditions are equivalent:
\begin{itemize}
\item[\(\rm i)\)] The direct system
\(\big(\{G_i\}_{i\in I},\{\varphi_{ij}\}_{i\leq j}\big)\)
is non-degenerate (in the sense of Definition \ref{def:non-deg_DS_CMSG}).
\item[\(\rm ii)\)] The direct limit of \(\big(\{G_i\}_{i\in I},\{\varphi_{ij}\}_{i\leq j}\big)\) in the category
of complete metric scalable groups exists. Moreover, it equals the metric completion of the metric limit of \(\big(\{G_i\}_{i\in I},\{\varphi_{ij}\}_{i\leq j}\big)\).
\end{itemize}
\end{theorem}
\begin{proof}
{\color{blue}\({\rm i)}\Longrightarrow{\rm ii)}\)} Suppose i) holds. Due to Lemma \ref{lem:cont_operations} \ref{it:inv_vs_Rx},  \ref{eq:inv_non-deg_lem} implies \ref{eq:Rx_non-deg_lem}. Hence, by Theorem \ref{thm:MSG_DL}, the direct system \(\big(\{G_i\}_{i\in I},\{\varphi_{ij}\}_{i\leq j}\big)\) has a direct limit \((G'/Z(\sfd'), \{\phi_i\}_{i\in I})\) in the category of metric scalable groups. By \ref{eq:inv_non-deg_lem} the inversion map is Cauchy-continuous on $ G' $, yielding the Cauchy-continuity of the inversion on $ G'/Z(\sfd') $. By Lemma \ref{lem:suff_cond_Cauchy_op}, the metric limit \(\big(G'/Z(\sfd'), \delta, \sfd\big)\) admits a completion $ (G,\delta, \sfd) $ as a metric scalable group. Denoting by $ \iota \colon G'/Z(\sfd') \hookrightarrow G$ the isometric embedding, we define for every  $ i\in I $ the map $ \varphi_i \colon G_i \to G $ by
\[
\varphi_i\coloneqq\iota \circ \phi_i.
\]
Since each $ \varphi_i $ is a morphism of scalable groups, $ (G,\delta, \sfd) $ is a target of $\big(\{G_i\}_{i\in I},\{\varphi_{ij}\}_{i\leq j}\big)$.

To show that $\big(\{G_i\}_{i\in I},\{\varphi_{ij}\}_{i\leq j}\big)$ has the universal property, let $ (H,\{\psi_i\}_{i\in I}) $ be any target in the category of complete metric scalable groups. Since $ H $ is also a target in the category of metric scalable groups, there exists a unique morphism $ \Phi' \colon G'/Z(\sfd') \to H $ of metric scalable groups. In particular, $ \Phi' $ is 1-Lipschitz, thus it can be uniquely extended to a 1-Lipschitz map $ \Phi \colon G \to H $. By approximation, we can see that \(\Phi\) is the unique morphism of metric scalable groups such that \(\psi_i=\Phi\circ\varphi_i\) holds for every \(i\in I\). This shows that
$\big(\{G_i\}_{i\in I},\{\varphi_{ij}\}_{i\leq j}\big)$ satisfies the universal property.
Therefore, ii) is achieved.\\
{\color{blue}\({\rm ii)}\Longrightarrow{\rm i)}\)} Suppose ii) holds. In particular, the metric completion of the metric limit $ G'/Z(\sfd') $ is a complete metric scalable group. Hence the dilation, the group operation and the inversion map are Cauchy-continuous. Consequently, the corresponding maps on $ G' $ are Cauchy-continuous. We conclude that \ref{eq:delta_non-deg_lem} and \ref{eq:inv_non-deg_lem} hold, thus i) is satisfied.
\end{proof}
\begin{remark}\label{rmk:density_claim}{\rm
It is not clear if a direct system \(\big(\{G_i\}_{i\in I},\{\varphi_{ij}\}_{i\leq j}\big)\) of complete metric scalable groups that does not satisfy \ref{eq:delta_non-deg_lem} or \ref{eq:inv_non-deg_lem} can have a direct limit.
Indeed, Theorem \ref{thm:DL_of_CMSG} only implies that for such a direct system the construction using metric limit fails. However, it is fairly straightforward to show that, whenever there exists a direct limit \(\big(G,\{\varphi_{i}\}_{i\in I}\big)\), then it holds that
\[
\bigcup_{i\in I}\varphi_i(G_i)\quad\text{ is dense in }G.
\]
Indeed, calling \(\widetilde G\) the closure of \(\bigcup_{i\in I}\varphi_i(G_i)\)
in \(G\), it is easy to check that \(\widetilde G\) is a complete metric scalable
subgroup of \(G\) and that \(\big(\widetilde G,\{\varphi_i\}_{i\in I}\big)\) is
a target. Denote by \(\iota\colon\widetilde G\to G\) the inclusion map,
which clearly is a morphism of CMSGs. Now fix any target
\((H,\{\psi_i\}_{i\in I})\). Call \(\Phi\colon G\to H\) the unique
morphism such that \(\psi_i=\Phi\circ\varphi_i\) for every \(i\in I\).
Then \(\widetilde\Phi\coloneqq\Phi\circ\iota\) is the unique morphism
\(\widetilde\Phi\colon\widetilde G\to H\) such that \(\psi_i=\widetilde\Phi\circ\varphi_i\)
for every \(i\in I\), which implies that actually \(\widetilde G=G\).
\fr}\end{remark}
Finally, we prove that when the scalable group $ G' $ is nilpotent, \ref{eq:Rx_non-deg_lem} and \ref{eq:inv_non-deg_lem} are equivalent.
\begin{lemma}
	\label{lem:nilpotent_inv_cauchy}
	Let \(G\) be a nilpotent group endowed with a left-invariant
	distance \(\sfd\). If $ R_x\colon G \to G $ is continuous for every $ x\in G $, then the inversion $ {\sf Inv} \colon G \to G$ is Cauchy-continuous.
\end{lemma}
\begin{proof}
    The proof is by induction on the nilpotency step $ s $ of $ G $. To prove the base case, assume that $ s=1 $, that is, $ G $ is abelian. Then for every $ x,y \in G $,
	\[
	\sfd(x^{-1},y^{-1}) = \sfd(e,xy^{-1}) = \sfd(y,x),
	\]
	which proves that $ {\sf Inv} $ is an isometry.
	
	Let then the nilpotency step $ s $ of $ G $ be arbitrary and assume that the claim is true for groups of nilpotency step $ s-1 $.  Denote by $ G^{(s)} $ the last element of the lower central series of $ G $, which is an abelian subgroup of $ G $. Then the quotient space $ G/G^{(s)} $ is a group of step $ s-1 $. Moreover, on $ G/G^{(s)} $ one may consider (see \cite{MZ_1955}, p.\ 36) a left-invariant distance $ \sfd' $ defined by 
	\[
	\sfd'([x],[y]) \coloneqq \inf_{h\in G^{(s)}} \sfd(x,hy)\quad
	\text{ for every }[x],[y]\in G/G^{(s)}\,.
	\]
	Let $ T\sus G $ be totally bounded and let us show that $ {\sf Inv} $ is uniformly continuous on $ T $, which would prove the claim by Lemma \ref{lemma:Cauchy_continuity_vs_moduli}. Since the projection map $ x \mapsto [x] $ is Lipschitz, the set $ [T] $ is totally bounded. Denoting by $ {\sf Inv}' $ the inversion on  $ G/G^{(s)} $, by induction assumption and Lemma \ref{lemma:Cauchy_continuity_vs_moduli} we have that $[{\sf Inv}(T)]= {\sf Inv}'([T]) $ is totally bounded as well. Fix $ \eps > 0 $ and let $ [x_1],\dots,[x_n]\in G/G^{(s)} $ be such that $ [{\sf Inv}(T)] \sus \bigcup_{i=1}^nB_{\eps/3}([x_i]) $. Then it is easy to show that, if $ U\coloneqq \bigcup_{i=1}^nB_{\eps/3}(x_i) \sus G $, then  $ [{\sf Inv}(T)] \sus [U] $.
	Since $ R_x $ is continuous for every $ x\in G $, we have
	\[
	\eta \coloneqq \inf_i \omega_{R_{x_i}}(\eps/3) > 0.
	\]
	Moreover, by Lemma \ref{lem:cont_operations} \ref{it:Rx_extension} it holds $\inf_{x\in U} \omega_{R_{x}}(\eps) \geq \eta. $
	
	Let finally $ x\in T $ and let us demonstrate that $ \omega_{{\sf Inv}}(x;\eps) \geq \eta $.	Indeed, now $ [x^{-1}]\in [{\sf Inv}(T)] \sus [U] $, which implies the existence of $ h \in G^{(s)} $ such that $ hx^{-1} \in U $. Recall that $ G^{(s)} $ is an abelian subgroup of $ G $, whence $ R_{h^{-1}} $ is an isometry. Consequently,
	\[
	\omega_{{\sf Inv}}(x;\eps)  \overset{\text{L}\ref{lem:cont_operations} \ref{it:inv_vs_Rx}}{=} \omega_{R_{x^{-1}}}(e;\eps) =  \omega_{R_{hx^{-1}}\circ R_{h^{-1}} }(e;\eps) \overset{(*)}{=} \omega_{R_{hx^{-1}}}(h^{-1};\eps) \overset{\text{L}\ref{lem:cont_operations} \ref{it:Rx_unif_cont}}{=} \omega_{R_{hx^{-1}}}(e;\eps) \geq \eta,
 	\]
 	where the identity $ (*) $ is a straightforward consequence of the fact that $ R_{h^{-1}} $ is an isometry.
\end{proof}
\begin{corollary}[Direct limits of CMSGs with bounded nilpotency step]
	\label{cor:nilpotent_DL}
Consider a direct system \(\big(\{G_i\}_{i\in I},\{\varphi_{ij}\}_{i\leq j}\big)\)
of complete metric scalable groups satisfying \ref{eq:delta_non-deg_lem} and \ref{eq:Rx_non-deg_lem}. Let us suppose that the nilpotency step of the groups $ G_i $ is uniformly bounded. Then it holds that the direct limit of \(\big(\{G_i\}_{i\in I},\{\varphi_{ij}\}_{i\leq j}\big)\) exists.
\end{corollary}
\begin{proof}
	Let $ s\in \N $ be such that every $ G_i $ is nilpotent of step less than $ s $ and fix $ x_1,\dots,x_s \in G' $. Then there exist $ k\in I $ with the following property: for every $ i= 1,\dots,s $, there exists $ y_i \in G_k $ with $ \varphi'_k(y_i) = x_i $. Since $ G_k $ is nilpotent of step less than $ s $, we have
	\[
	[x_1,\dots,[x_{s-1},x_s]\dots] = [\varphi'_k(y_1),\dots,[\varphi'_k(y_{s-1}),\varphi'_k(y_s)]\dots] =  \varphi'_k([y_1,\dots,[y_{s-1},y_s]\dots]) = e'.
	\] 
	This proves that $ G' $ is nilpotent. Given \ref{eq:Rx_non-deg_lem}, by Lemma \ref{lem:nilpotent_inv_cauchy} the condition \ref{eq:inv_non-deg_lem} holds. The claim follows then from Theorem \ref{thm:DL_of_CMSG}.
\end{proof}
	\section{Infinite-dimensional Carnot groups}
	\label{s:inf-dim_Carnots}
	
	The aim of this section is to investigate direct systems of Carnot groups and their relations to the notion of infinite-dimensional Carnot group. We recall first the relevant concepts for infinite-dimensional Carnot groups from \cite{LDLM19}. Direct limits of infinite-dimensional Carnot groups are discussed in Section \ref{sec:DL_of_infdim_Carnots}.
	
	A \emph{Carnot group} is, by definition, a connected and simply connected Lie group whose Lie algebra admits a stratification. Given a stratified Lie algebra, one naturally defines a family of dilation automorphisms on the group via the vector space scalings on the first layer of the stratification. Therefore, any Carnot group is a topological scalable group, and becomes a complete metric scalable group when equipped with a homogeneous distance.	Moreover, we say that a topological scalable group \emph{has Carnot group structure} if it is isomorphic as a topological scalable group to some Carnot group.
	
	Separable Banach spaces have the useful property of admitting
	a countable, dense collection of finite-dimensional vector subspaces. This kind of collection has a natural generalization in the Carnot setting.
	\begin{definition}[Filtration by Carnot subgroups]
		We say that a topological scalable group $G$ is {\em filtrated by Carnot subgroups}
		if there exists a sequence $(N_m)_m$, $m\in \N$, of topological scalable subgroups of $G$ such that each $N_m$ has Carnot group structure, $N_m<N_{m+1}$, and
		$G$ is the closure of $\bigcup_{m\in\N}N_m$.
		In this case, we say that
		the sequence $(N_m)_m$, $m\in \N$,   is {\em a filtration by Carnot subgroups of the topological scalable group $G$}.
	\end{definition}
	We then define a non-commutative analogue of separable Banach spaces.
	\begin{definition}[Infinite-dimensional Carnot group]
	\label{def:inf-dim_Carnot}
		A complete metric scalable group admitting a filtration by Carnot subgroups is called an \emph{infinite-dimensional Carnot group}.
	\end{definition}
	It is often convenient to use an equivalent algebraic criterion for a topological scalable group to have filtrations. Similarly to classical Carnot groups, those elements in which the dilation map is a one-parameter subgroup play a special role for the geometry.
	\begin{definition}[First layer]
		We define for a scalable group $G$ its {\em first layer} as 
		\[V_1(G) \coloneqq\big\{ x\in G\;\big|\;\delta_{t+s}(x) = \delta_t(x)\delta_s(x)\text{ for every } t,s \in \R\big\}. \]
		We say then that the map $ t\in \R \mapsto \delta_t (x)\in G$ is a \emph{one-parameter subgroup}.
	\end{definition}
	If $ G $ is a group and $ A \subset G $ is a subset, we denote by $ \langle A \rangle $ the group generated by $ A $, i.e.,  the collection of all finite products of elements of $ A $ and of their inverses.
	If $ G $ is a scalable group equipped with a topology, we denote by $ \langle A \rangle_{SC}$ the closure of the group generated by $ \{\delta_t(a) \mid a\in A, t\in \R \} $. Observe that the first layer $ V_1(G) $ of a scalable group is invariant under dilations: $ \delta_t(V_1(G)) = V_1(G) $ for all $ t\in \R\setminus \{0\} $. Therefore, $G =  \langle V_1(G) \rangle_{SC} $ if and only if $ \langle V_1(G) \rangle $ is dense in $ G $.

	It turns out that filtrations of topological scalable groups are in close connection with the generating first layer. The following result is proven in \cite[Proposition 2.1]{LDLM19}.
	\begin{proposition}[Alternative characterization of admitting a filtration]\label{prop:filtration_equiv_to_generating_V1}
		Let $G$ be a topological scalable group. Then the following are equivalent:
		\begin{itemize}
			\item[\(\rm i)\)] $ G $ admits a filtration by Carnot subgroups;
			\item[\(\rm ii)\)] there exists a countable set $ A\subset V_1(G) $ such that $A $ generates $ G $ as a topological scalable group and $ \langle \Omega \rangle$ is nilpotent for every finite subset $ \Omega \subset A $.
		\end{itemize}
	\end{proposition}
	We stress that if a topological scalable group has generating first layer, then having elements in which the dilation map fails to be a one-parameter subgroup is the characterizing feature of non-commutativity. We formulate this observation in the following proposition.
	\begin{proposition}\label{prop:V1(G)=G_equiv_Banach}
		Let $ (G, \delta, \sfd) $ be a complete metric scalable group such that $ \langle V_1(G) \rangle_{SC} = G$. Then the following are equivalent:
		\begin{enumerate}[label=\rm\roman{*})]
			\item \label{it:V1} $ V_1(G) = G $;
			\item \label{it:abelian} $ G $ is abelian;
			\item \label{it:banach} $ G $ is a Banach space.
		\end{enumerate}
	\end{proposition}
	\begin{proof}
		We show first that \ref{it:V1} and \ref{it:abelian} are equivalent. We start by proving the following claim:
		\begin{equation}\label{eq:V1_inverse_is_dilation}
		\text{If } x \in V_1(G), \text{ then }\delta_{-1}(x) =x^{-1}.
		\end{equation}
		The proof is a simple computation:
		\[
		x\,\delta_{-1}(x) = \delta_1(x)\delta_{-1}(x) = \delta_{1-1}(x) = \delta_0(x) = e.
		\]
		Suppose now that \ref{it:V1} holds and take $ x,y\in G $. Then $ xy\in V_1(G) $ and \ref{it:abelian} follows from \eqref{eq:V1_inverse_is_dilation} as
		\[
		[x,y] = xyx^{-1}y^{-1}=xy\delta_{-1}(xy) = xy(xy)^{-1}=e.
		\]
		Assume then that \ref{it:abelian} holds and let $ x = x_1\dots x_n $ where each $ x_i \in V_1(G)$. Then for every $ t,s\in \R $,
		\begin{align*}
		\delta_t(x)\delta_s(x) &= \delta_t(x_1)\dots\delta_t(x_n)\delta_s(x_1)\dots\delta_s(x_n)=\delta_t(x_1)\delta_s(x_1)\dots\delta_t(x_n)\delta_s(x_n) \\
		&= \delta_{t+s}(x_1)\dots \delta_{t+s}(x_n) = \delta_{t+s}(x).
		\end{align*}
		Since $ \delta $ is continuous and $ \langle V_1(G) \rangle $ is dense in $ G $, we obtain \ref{it:V1}.
		
		The fact that \ref{it:banach} implies \ref{it:V1} and \ref{it:abelian} is immediate. Regarding the opposite direction, recall that a real topological vector space is an abelian topological group $ (V,+) $ equipped with a continuous scalar multiplication $\R\times V\ni(t,v) \mapsto tv\in V $ satisfying:
		\begin{itemize}
			\item[1.] $ 1v = v $;
			\item[2.] $ t(sv) = (ts)v $;
			\item[3.] $ tu + tv = t(u+v) $ and
			\item[4.] $ tv + sv = (t+s)v $.
		\end{itemize}
		Assuming \ref{it:V1} and \ref{it:abelian}, it is easy to check that $ (G,\delta) $ is a topological vector space with the scalar multiplication $ (t,x) \mapsto \delta_t(x) $. Moreover, letting $ \|x\|\coloneqq \sfd(e,x) $ defines a norm on $ G $. As we assumed $ G $ to be complete with respect to $ \sfd $, we conclude that $\big(G, \delta, \|\cdot\|\big) $ is a Banach space.
	\end{proof}
	Notice that in Proposition \ref{prop:V1(G)=G_equiv_Banach} the assumption of having generating first layer cannot be removed. Indeed, consider $ G = (\R,+) $ with the dilations given by $ \delta_\lambda(r) = \lambda^2 r $ and the distance $ \sfd = |\cdot|^{\frac{1}{2}} $. Then $ (G,\delta,\sfd) $ is an abelian complete metric scalable group with $ V_1(G)= \{0\} $. The group $ G $ can even be made geodesic, see \cite[p.\ 21]{M20}.

	\subsection{Direct limits of Carnot groups}
	The following proposition shows that any complete metric scalable group admitting a filtration is obtained as a direct limit of Carnot groups. In this simple case a direct proof of the universal property is straightforward. We nevertheless prove Proposition \ref{prop:infdim_Carnot_is_DL} by applying the more general Theorem \ref{thm:DL_of_CMSG} for the sake of an example.
	\begin{proposition}[Infinite-dimensional Carnot groups as direct limits]
	\label{prop:infdim_Carnot_is_DL}
		Let \((G,\delta,\sfd)\) be a complete metric scalable group admitting a
		filtration \((N_m)_m\) by Carnot groups. For any \(m,n\in\N\) with \(m\leq n\),
		we denote by \(\iota_{mn}\colon N_m\hookrightarrow N_n\) and
		\(\iota_m\colon N_m\hookrightarrow G\) the inclusion maps. Then
it holds that the direct system of complete metric scalable groups
\(\big(\{N_m\}_{m\in\N},\{\iota_{mn}\}_{m\leq n}\big)\) is non-degenerate 		
		and its direct limit is given by \(\big(G,\{\iota_m\}_{m\in\N}\big)\).
	\end{proposition}
	\begin{proof}
		First of all, \(\iota_{mn}\) and \(\iota_m\) are morphisms of
		metric scalable groups satisfying \(\iota_m=\iota_n\circ\iota_{mn}\) for
		all \(m,n\in\N\) with \(m\leq n\), thus \(\big(G,\{\iota_m\}_{m\in\N}\big)\)
		is a target of \(\big(\{N_m\}_{m\in\N},\{\iota_{mn}\}_{m\leq n}\big)\). 
		Now the direct limit of \(\big(\{N_m\}_{m\in\N},\{\iota_{mn}\}_{m\leq n}\big)\) in the category of scalable groups is the union
	 	\(G'=\bigcup_{m\in\N}N_m \subset G\). Moreover, we have $ Z(\sfd')=\{e\} $ (recall \eqref{eq:inf-pseudodist-zeroset} for the definition), so the metric limit of \(\big(\{N_m\}_{m\in\N},\{\iota_{mn}\}_{m\leq n}\big)\) equals $ (G', \delta, \sfd) $.
	 	
	 	Observe that, since $ G $ is complete, the dilation and the inversion map are Cauchy-continuous on $ G $ and, therefore, on $ G' $.
	 	This shows that the direct system of CMSGs
	 	$\big(\{N_m\}_{m\in\N},\{\iota_{mn}\}_{m\leq n}\big)$ is non-degenerate. Hence, by Theorem \ref{thm:DL_of_CMSG}, the direct limit of  
	 	$\big(\{N_m\}_{m\in\N},\{\iota_{mn}\}_{m\leq n}\big)$
	 	exists in the category of complete metric scalable groups, and it equals the metric completion of $ G' $. Since $ G' $ is dense in \(G\) by assumption, we conclude that \(\big(G,\{\iota_m\}_{m\in\N}\big)\) is the direct limit of $\big(\{N_m\}_{m\in\N},\{\iota_{mn}\}_{m\leq n}\big)$ in the category of complete metric scalable groups, as claimed.
	\end{proof}

	 \label{sec:DL_of_infdim_Carnots}
	Our next aim is to study in which circumstances a direct limit of (infinite-dimensional) Carnot groups exists in the category of CMSGs and when the limit is an infinite-dimensional Carnot group. Because of the generating first layer, the condition \ref{eq:delta_non-deg_lem} on the continuity of the dilation is automatically satisfied, as we will show next.
	\begin{lemma}\label{lemma:generating_V1_implies_delta_Cauchy}
		Let $ (G,\delta) $ be a scalable group equipped with a compatible distance $ \sfd $. If $ R_x \colon G \to G$ is continuous for every $ x\in G $ and $G = {\langle V_1(G)\rangle}_{SC} $, then $ \delta\colon\R\times G\to G$ is Cauchy-continuous.
	\end{lemma}
	\begin{proof}
Call \(S\) the set of \(x\in G\) such that \(\R\ni t\mapsto\delta(t,x)\in G\)
is continuous. By Lemma \ref{lem:cont_operations} \ref{it:delta_Cauchy}, in
order to prove the statement it suffices to show that \(S=G\).
Note first that if $ x \in V_1(G) $, then
		\[
		\sfd(\delta_t(x),\delta_s(x)) = \sfd(e,\delta_{s-t}(x)) = |s-t|\sfd(e,x)
		\]
		and the map $ \delta(\cdot,x) $ is an isometric embedding of $ \R $ into $ G $, so that \(V_1(G)\subseteq S\). Let then $ x,y\in S $.
Since the right translation by $ \delta_t(y) $ is continuous, we have that
		\begin{align*}
		\sfd\big(\delta_s(xy),\delta_t(xy)\big) &\leq \sfd\big(\delta_s(x)\delta_s(y),\delta_s(x)\delta_t(y)\big)+\sfd\big(\delta_s(x)\delta_t(y),\delta_t(x)\delta_t(y)\big) \\
		&= \sfd\big(\delta_s(y),\delta_t(y)\big)+\sfd\big(\delta_s(x)\delta_t(y),\delta_t(x)\delta_t(y)\big) \to 0 \quad\text{ when } s\to t,
		\end{align*}
which proves that $ \delta(\cdot,xy) $ is continuous and thus \(xy\in S\).
A direct computation yields \(x^{-1}\in V_1(G)\) for every
\(x\in V_1(G)\), which shows that \(\langle V_1(G)\rangle\) is the set of
finite products of elements of \(V_1(G)\), whence it follows that $\langle V_1(G)\rangle\subseteq S$.
Finally, we want to exploit the density of \(\langle V_1(G)\rangle\) in \(G\)
to conclude that \(S=G\). To do so, observe that for any \(x,y\in G\) and
\(a\in(0,+\infty)\) it holds that
\[
\sup_{t\in[-a,a]}\sfd\big(\delta_t(x),\delta_t(y)\big)=
\sup_{t\in[-a,a]}|t|\,\sfd(x,y)\leq a\,\sfd(x,y).
\]
Therefore, given \(x\in G\) and \((x_n)_n\subseteq\langle V_1(G)\rangle\)
with \(\lim_n\sfd(x_n,x)=0\), we have that
\(\delta(\cdot,x_n)\to\delta(\cdot,x)\) uniformly on compact sets.
Being \(\delta(\cdot,x_n)\) continuous for every \(n\in\N\), we infer
that \(\delta(\cdot,x)\) is continuous as well, thus proving that
\(x\in S\). Consequently, the statement is achieved.
	\end{proof}
	The proof of Theorem \ref{thm:DL_of_infdim_Carnots} is obtained as a corollary of the following result.
	\begin{theorem}[Direct limits of infinite-dimensional Carnot groups]
		\label{thm:DL_of_infdim_Carnots_proof}
		If \(\big(\{G_i\}_{i\in I},\{\varphi_{ij}\}_{i\leq j}\big)\) is a direct system of infinite-dimensional Carnot groups in the category of CMSGs satisfying \ref{eq:inv_non-deg_lem}, then it has a direct limit. If, in addition, $ I $ is countable and each $ G_i $ is nilpotent, then the limit is an infinite-dimensional Carnot group.
	\end{theorem}
	\begin{proof}
		We start with the following algebraic observation.
		Since each $ G_i $ is an infinite-dimensional Carnot group, by Proposition \ref{prop:filtration_equiv_to_generating_V1} there exists for each $ i\in I $ a countable set $ A_i \subset V_1(G_i) $ such that $ \langle A_i \rangle_{SC} = G_i $ and $ \langle \Omega_i \rangle $ is nilpotent for every finite subset $ \Omega_i \subset A_i $. Consider then the direct limit $ G' $ of \(\big(\{G_i\}_{i\in I},\{\varphi_{ij}\}_{i\leq j}\big)\) in the category of scalable groups equipped with the infimum-pseudodistance $ \sfd' $. Observe that $ \varphi'_i(V_1(G_i))\subset V_1(G) $ for every $ i\in I $, since $ \varphi'_i $ are morphisms of scalable groups.
Hence the set $ A \coloneqq \bigcup_{i\in I}\varphi'_i(A_i) $ is contained in $ V_1(G') $. Recall that the morphisms $ \varphi'_i $ are continuous with respect to the topology induced by the infimum-pseudodistance $ \sfd' $ on $ G' $. Then $ \varphi'_{i}(\langle A_i \rangle_{SC})\sus \langle \varphi'_i( A_i )\rangle_{SC}   $ for every $ i\in I $, and consequently
		\begin{equation}\label{eq:G'_generated}
		G' \subseteq \bigcup_{i\in I}\varphi'_i(G_i) = \bigcup_{i\in I}\varphi'_i(\langle A_i \rangle_{SC}) \sus \bigcup_{i\in I}\langle \varphi'_i( A_i )\rangle_{SC} \subseteq  \big\langle\bigcup_{i\in I} \varphi'_i( A_i )\big\rangle_{SC} = \langle A \rangle_{SC}.
		\end{equation}
				
		Equation \eqref{eq:G'_generated} shows that, in particular, $G' = \langle V_1(G') \rangle_{SC}$. Moreover, by \ref{eq:inv_non-deg_lem} and Lemma \ref{lem:cont_operations} \ref{it:inv_vs_Rx}, the right translations in $ G' $ are continuous. Then by Lemma \ref{lem:metric_limit_is_msg}, the scalable group $\big(G'/Z(\sfd'),\delta\big)$ together with the compatible infimum-distance $ \sfd $ is well-defined. We are now in a position to apply Lemma \ref{lemma:generating_V1_implies_delta_Cauchy}, which gives that the dilation map $ \delta $ is Cauchy-continuous on $ G'/Z(\sfd') $. Hence \ref{eq:delta_non-deg_lem} is satisfied. Since also \ref{eq:inv_non-deg_lem} is assumed to hold, by Theorem \ref{thm:DL_of_CMSG} the direct limit \(\big(G,\{\varphi_i\}_{i\in I}\big)\) in the category of complete metric scalable groups exists, and it equals the metric completion of the metric limit $\big(G'/Z(\sfd'),\delta, \sfd\big)$. This proves the first part of the claim.
		
		Finally, assume that each $ G_i $ is nilpotent and $ I $ is countable. Then the set $ A = \bigcup_{i\in I}\varphi'_{i}(A_i)  \subset V_1(G')$ introduced above is countable. Since the image of $ G' $ is dense in $ G $, \eqref{eq:G'_generated} guarantees that $ G = \langle A \rangle_{SC} $. Then, in order to prove that $ G $ is an infinite-dimensional Carnot group, by Proposition \ref{prop:filtration_equiv_to_generating_V1} it is enough to show that $ \langle \Omega \rangle $ is nilpotent for every finite subset $ \Omega \subset A $ . Now given such a set $ \Omega $, for every $ a \in \Omega $ fix $ i \in I $ such that there exists $ a_i \in A_i $ with $ a = \varphi_i(a_i) $. Let $ \ell \in I $ be such that, for every $ a\in \Omega $ and $ i\in I $ associated to $ a $, we have $ \ell > i $. Then, for every $ a \in \Omega $,
		\[
		a = \varphi_i(a_i) = \varphi_\ell \circ \varphi_{i\ell}(a_i) \in \varphi_\ell(G_\ell).
		\]
		Hence $ \langle \Omega \rangle \subset \varphi_\ell(G_\ell)$. Since each $ G_\ell $ is nilpotent by assumption,  $ \varphi_\ell(G_\ell) $ is a nilpotent subgroup of $ G $. We conclude that $ \langle \Omega \rangle $ is nilpotent and $ G $ is an infinite-dimensional Carnot group, which we were aiming to show.
	\end{proof}
We do not know if the assumption on nilpotency of the groups $ G_i $ in Theorem \ref{thm:DL_of_infdim_Carnots_proof} can be removed; we do not have a counterexample.
\medskip

Theorem \ref{thm:DL_of_infdim_Carnots_proof} gives, together with Corollary \ref{cor:nilpotent_DL}, the following result.
	\begin{corollary}\label{cor:DL_Carnot}
		Let \(\big(\{G_i\}_{i\in I},\{\varphi_{ij}\}_{i\leq j}\big)\) be a 
countable direct system of Carnot groups in the category of CMSGs such that the nilpotency step of the groups $ G_i $ is uniformly bounded. If \ref{eq:Rx_non-deg_lem} is satisfied, then the direct limit of \(\big(\{G_i\}_{i\in I},\{\varphi_{ij}\}_{i\leq j}\big)\) exists in the category of CMSGs and it is an infinite-dimensional Carnot group.
	\end{corollary}
It would be interesting to know whether a uniform bound on the nilpotency steps of Carnot groups $ G_i $ automatically guarantees condition \ref{eq:Rx_non-deg_lem}. In Proposition \ref{prop:degenerate_DS} we give an example of a direct system of Carnot groups $ G_i $ that does not satisfy \ref{eq:Rx_non-deg_lem} and for which the nilpotency steps of $ G_i $ grow unlimitedly.
	
	\subsection{Example of a degenerate direct system of Carnot groups}
	\label{ss:ex_deg_DS}	
	We now give an example of a direct system of Carnot groups that does not satisfy \ref{eq:Rx_non-deg_lem} (see Proposition  \ref{prop:degenerate_DS}). Below we denote by $ \mathbb F_{2,k} $ the Free Lie group of rank 2 and step $ k $ with the canonical projections $ \pi^l_k\colon  \mathbb F_{2,l}\to  \mathbb F_{2,k} $ for every  $l, k\in \N $ with $ k\leq l $. For each $ k\in \N $, fix a basis $ X_k,Y_k $ for the first layer of the Lie algebra such that $ \pi^l_k(x_l) = x_k $, and $ \pi^l_k(y_l)=y_k $ for every $l,k\in \N$, $ k\leq l $, where we denote $ x_k \coloneqq \exp(X_k) $ and $ y_k \coloneqq \exp(Y_k) $. We also fix a norm on each horizontal layer for which $ \|X_k\| = \|Y_k\| = 1 $ and consider the corresponding Carnot--Carath\'{e}odory distance $ \sfd^k_{\mathbb F} $ on $ \mathbb F_{2,k} $.
	
	In the following lemma we put together several results in \cite{LDZ19}.
\begin{lemma}\label{lem:Free_gp_noncont}
	For each $k\in \N $, let $ \mathbb F_{2,k} $ be the free Lie group of rank 2 and step $ k $ with generators $ x_k$ and $y_k $ as described above. Then for every $ \eps > 0 $, it holds
	\[
	\lim_{k\to \infty} \sfd^k_{\mathbb F}(x_k, \delta_\eps(y_k)x_k) > 2.
	\]
\end{lemma}
\begin{proof}	
	Fix any $ \eps > 0 $ and consider the curve $ \gamma\colon [0,2+\eps]\to \R^2 $,
	\[
	\gamma(t) = \begin{cases}
	(-t,0),&\quad t\in [0,1], \\
	(-1,t-1),&\quad t\in [1,1+\eps], \\
	(t-(2+\eps),\eps)&\quad t\in [1+\eps,2+\eps]. 
	\end{cases}
	\]
	By \cite[Lemma 4.1]{LDZ19}, the curve $ \gamma $ admits a unique lift $ \gamma_k $ to $ \mathbb F_{2,k} $ for every $ k\in \N $. It follows from the uniqueness of the lift that $ \gamma_k(2+\eps) = x^{-1}_k\delta_\eps(y_k)x_k $ for every $ k\in \N $.  Let $  \big(\mathbb F_{2,\infty}, \{\pi^\infty_k \}_{k\in \N} \big) $ be the inverse limit of the inverse system $\big(\{ \mathbb F_{2,k}\}_{k\in \N},\{\pi^l_{k}\}_{k\leq l}\big)$ in the category of scalable groups equipped with the supremum-metric $ \sfd' $ (see \eqref{eq:sup-metric} for the definition). By \cite[Lemma 2.7]{LDZ19}, the curve $ \gamma $ admits a unique lift $ \eta $ to the group $  \mathbb F_{2,\infty} $ satisfying $ \pi^\infty_k(\eta(2+\eps)) = x^{-1}_k\delta_\eps(y_k)x_k $ and $ \pi^\infty_k(\eta(0))=e_k $ for every $ k\in \N $. Moreover, the curve $ \eta $ has length $ L(\eta) = L(\gamma) = 2+\eps $.
	
	According to \cite[Theorem 4.2]{LDZ19}, the inverse limit $ ( \mathbb F_{2,\infty},\sfd') $ is a metric tree. In particular, the injective curve $ \eta $ is a geodesic, which implies that $  \sfd'(e_\infty,\eta(2+\eps))=\sfd'(\eta(0),\eta(2+\eps)) = L(\eta) = 2+\eps$. Hence, by definition of $ \sfd' $, we obtain
	\[
	\lim_{k\to \infty} \sfd^k_{\mathbb F}(e_k, x_k^{-1}\delta_\eps(y_k)x_k) = 	\lim_{k\to \infty} \sfd^k_{\mathbb F}\big(\pi^\infty_k(e_\infty),\pi^\infty_k(\eta(2+\eps))\big) = \sfd'(e_\infty,\eta(2+\eps)) = 2+\eps.
	\]
	The claim follows then from the left-invariance of each distance $ \sfd^k_{\mathbb F} $.
\end{proof}

We now define a stratified Lie algebra $ \mathfrak g_i $ of rank $ i+1 $ and step $ i $ for every $ i\in \N $ as follows: we fix a basis $ \{X_i,Y_i^1,\dots,Y_i^i\} $ for the first layer of $ \mathfrak g_i $, where each pair $ \{X_i,Y^k_i\} $ generates the free Lie algebra of rank 2 and step $ k $, and where all the other brackets are zero. Observe that each $ \mathfrak g_i $ is indeed a Lie algebra: the bracket $ [Z_1,[Z_2,Z_3]] =0 $ for all distinct basis elements $ Z_1,Z_2,Z_3 \in \{X_i,Y_i^1,\dots,Y_i^i\} $, and hence the Jacobi identity is satisfied. We also fix a norm on the horizontal layer of $ \mathfrak g_i $ that gives unit length to each basis vector.

Let us denote by $ G_i $ the Carnot group whose Lie algebra is $ \mathfrak g_i $ and let $ \sfd_i $ be the Carnot--Carath\'{e}odory distance on $ G_i $. Denote also $ x_i \coloneqq \exp(X_i) $ and $ y_i^k \coloneqq \exp(Y_i^k) $ for every $ i\in \N $ and $ k= 1,\dots,i $. Observe that there exist isometric embeddings $ \iota_{ij} \colon G_i \to G_j $ for $ i\leq j $ satisfying $ \iota_{ij}(x_i) = x_j $ and $ \iota_{ij}(y_i^k) = y_j^k $ for every $ k= 1,\dots,i $. Moreover, each $ G_i $ contains an isometric copy of every $ \mathbb{F}_{2,k} $ for $ k= 1,\dots,i $.
\begin{proposition}[A degenerate direct system of Carnot groups]
	\label{prop:degenerate_DS}
Let $ G_i $ and  $ \iota_{ij} \colon G_i \to G_j $ be as above for every $ i,j\in \N $, $ i\leq j $. Then the direct system \(\big(\{G_i\}_{i\in \N},\{\iota_{ij}\}_{i\leq j}\big)\) does not satisfy \ref{eq:Rx_non-deg_lem}.
\end{proposition}
\begin{proof}
	Let $ \big(G',\{\iota_i\}_{i\in \N}\big)$ be the direct limit of \(\big(\{G_i\}_{i\in \N},\{\iota_{ij}\}_{i\leq j}\big)\) in the category of scalable groups. Observe that $ \iota_i(x_i) = \iota_j(x_j) $ for every $ i,j\in \N $ and $ \iota_i(y_i^k)= \iota_j(y_j^k) $ for every $ k\in \N $ and $ i,j\geq k $. Let $ x\in G' $ be the element satisfying $ x=\iota_i(x_i) $ for every $ i\in \N $. We are going to show that $ R_x $ is not continuous at the identity with respect to the pseudometric $ \sfd' $ defined in \eqref{eq:inf-pseudodist}, which would prove that  \ref{eq:Rx_non-deg_lem} is not satisfied. 
	
	Let $ \eps > 0 $. Since every $ G_i $ contains an isometric copy of $ \mathbb F_{2,i} $ with generators $ x_i, y_i^i $, by Lemma \ref{lem:Free_gp_noncont} there exists $ k \in \N $ such that
	\[
	\sfd_k(x_k, \delta_\eps(y_k^k)x_k) > 2.
	\]
	Moreover, we have $ \sfd_k(e_k,\delta_\eps(y_k^k))=\eps $. Denote $ y_\eps \coloneqq \iota_k(\delta_\eps(y_k^k)) \in G' $. Since the morphisms $ \iota_{ij} $ are isometric embeddings, it follows from the definition of $ \sfd' $ that $ \sfd'(e',y_\eps) = \eps $ and
	\[
	\sfd'(\iota_k(x_k),y_\eps\, \iota_k(x_k)) = \sfd'(x,y_\eps x) > 2.
	\]
	Since $ \eps $ was arbitrary, this proves that $ R_x $ is not continuous at the identity, as claimed.
\end{proof}
\section{A Rademacher-type theorem for direct limits of CMSGs}
\label{s:Rademacher}
As already described in the Introduction, in this section
we study those CMSGs satisfying a suitable form of Rademacher
theorem. More specifically, we will show that such a class of
CMSGs is stable under taking direct limits,
thus generalizing one of the main results of \cite{LDLM19}.
\medskip

To begin with, we define what is a `null set' in a CMSG.
At this level of generality, it is not clear whether this
family of null sets is actually the family of negligible sets
of some measure. 
\begin{definition}[Null sets in a CMSG]\label{def:family_null_sets}
Let \(G\) be a complete metric scalable group. Then by \emph{family of null sets}
in \(G\) we mean a left-invariant \(\sigma\)-ideal \(\mathcal N\)
of the Borel \(\sigma\)-algebra \(\mathscr B(G)\), i.e.,
\begin{itemize}
\item[\(\rm i)\)] \(\mathcal N\) is a subset of the Borel \(\sigma\)-algebra
\(\mathscr B(G)\), containing the empty set.
\item[\(\rm ii)\)] If \((N_n)_{n\in\N}\subseteq\mathcal N\),
then \(\bigcup_{n\in\N}N_n\in\mathcal N\).
\item[\(\rm iii)\)] If \(N\in\mathcal N\) and \(N'\in\mathscr B(G)\) satisfy
\(N'\subseteq N\), then \(N'\in\mathcal N\).
\item[\(\rm iv)\)] \(gN\in\mathcal N\) for every \(g\in G\) and \(N\in\mathcal N\).
\end{itemize}
\end{definition}
Let us consider a countable family \(\{(G_i,\mathcal N_i)\}_{i\in I}\),
where \(\{G_i\}_{i\in I}\) is a direct system of CMSGs admitting a direct
limit \(G\), while the sets \(\{\mathcal N_i\}_{i\in I}\) are families
of null sets. Then, as we are going to show, there is a natural way
to define a family \(\mathcal N\) of null sets in \(G\). The construction
is inspired (and extends) the analogous ones in \cite{Aronszajn1976}
and \cite{LDLM19}.
\begin{lemma}
Let \(\big(\{G_i\}_{i\in I},\{\varphi_{ij}\}_{i\leq j}\big)\) be a direct
system of complete metric scalable groups (where the directed set \(I\) is
at most countable) having direct limit \(\big(G,\{\varphi_i\}_{i\in I}\big)\)
in that category. Suppose that each space \(G_i\) is endowed with a family of null sets \(\mathcal N_i\).
Then the set \(\mathcal N\subseteq\mathscr B(G)\), given by
\begin{equation}\label{eq:def_N}
\mathcal N\coloneqq\bigg\{\bigcup_{i\in I}N_i\;\bigg|\;
N_i\in\mathscr B(G)\text{ and }\varphi_i^{-1}(qN_i)\in\mathcal N_i
\text{ for every }i\in I\text{ and }q\in G\bigg\},
\end{equation}
is a family of null sets in \(G\).
\end{lemma}
\begin{proof}
Let us prove that \(\mathcal N\) satisfies the four conditions in Definition
\ref{def:family_null_sets}:\\
{\color{blue}i)} Trivially, \(\emptyset\in\mathcal N\subseteq\mathscr B(G)\).\\
{\color{blue}ii)} Fix any
\(\{N^n_i\,:\,n\in\N,\,i\in I\}\subseteq\mathscr B(G)\) such
that \(\bigcup_{i\in I}N^n_i\in\mathcal N\) for every \(n\in\N\). Let us define
the set \(N\in\mathscr B(G)\) as \(N\coloneqq\bigcup_{i\in I}\bigcup_{n\in\N}N^n_i\).
Given any \(i\in I\) and \(q\in G\), we have that
\[
\varphi_i^{-1}\bigg(q\bigcup_{n\in\N}N^n_i\bigg)=
\varphi_i^{-1}\bigg(\bigcup_{n\in\N}q N^n_i\bigg)=
\bigcup_{n\in\N}\varphi_i^{-1}(q N^n_i)\in\mathcal N_i,
\]
which proves that \(N\in\mathcal N\), as required.\\
{\color{blue}iii)} Fix any
\(N\in\mathcal N\) and \(N'\in\mathscr B(G)\) such that
\(N'\subseteq N\), say \(N=\bigcup_{i\in I}N_i\). Given \(i\in I\)
and \(q\in G\), we have that \(N_i\cap N'\in\mathscr B(G)\) and
\(\varphi_i^{-1}\big(q(N_i\cap N')\big)\subseteq\varphi_i^{-1}(q N_i)\in\mathcal N_i\),
whence \(N_i\cap N'\in\mathcal N_i\) as well. This grants that
\(N'=\bigcup_{i\in I}N_i\cap N'\in\mathcal N\), as required.\\
{\color{blue}iv)} Fix any
\(N=\bigcup_{i\in I}N_i\in\mathcal N\) and \(g\in G\). Since
\(\varphi_i^{-1}\big(q(g N_i)\big)=\varphi_i^{-1}\big((qg)N_i\big)\in\mathcal N_i\)
for every \(i\in I\) and \(q\in G\), thus  accordingly
\(gN=\bigcup_{i\in I}g N_i\in\mathcal N\), as required.
\end{proof}
Next we introduce the notion of G\^{a}teaux differential in the sense of Pansu, whose formulation is taken from \cite{LDLM19}.
\begin{definition}[G\^{a}teaux differentiability]\label{def:Gateaux_diff}
Let \((G,\delta,\sfd)\), \((G',\delta',\sfd')\) be metric scalable groups.
Given \(f\colon G\to G'\) and \(p\in G\), we say that \(f\)
is \emph{G\^{a}teaux differentiable} at \(p\) if the following hold:
\begin{itemize}
\item[\(\rm i)\)] For every element \(g\in G\), we have that
\[
\exists\,\d_p f(g)\coloneqq\lim_{\lambda\to 0}
\delta'_{1/\lambda}\Big(f(p)^{-1}f\big(p\,\delta_\lambda(g)\big)\Big)\in G'.
\]
\item[\(\rm ii)\)] The resulting function \(\d_p f\colon G\to G'\) is a
continuous group morphism.
\end{itemize}
We say that \(\d_p f\) is the \emph{G\^{a}teaux differential} of \(f\) at \(p\).
Moreover, let us define
\[
{\rm ND}(f)\coloneqq\big\{p\in G\;\big|\;
f\text{ is not G\^{a}teaux differentiable at }p\big\}.
\]
We say that \({\rm ND}(f)\subseteq G\) is the \emph{non-differentiability set} of \(f\).
\end{definition}
\begin{remark}{\rm
If \(f\) is G\^{a}teaux differentiable at \(p\), then its G\^{a}teaux differential
\(\d_p f\colon G\to G'\) is a Lipschitz morphism of scalable groups.
Indeed, it is well-known (and easy to prove) that \(\d_p f\) is a morphism
of scalable groups, while Proposition \ref{prop:morph_cont_vs_Lip} grants that \(\d_p f\)
is Lipschitz.
\fr}\end{remark}
We say that a complete metric scalable group \(G\) together with some family \(\mathcal N\) of null sets has the Rademacher property if every Lipschitz function
on \(G\) is `\(\mathcal N\)-almost everywhere' G\^{a}teaux differentiable.
More precisely:
\begin{definition}[Rademacher property]\label{def:Rademacher_prop}
Let \((G,\mathcal N)\) be a complete metric scalable group endowed
with a family of null sets. Then \((G,\mathcal N)\) is said to
\emph{have the Rademacher property} provided it holds
\[
{\rm ND}(f)\in\mathcal N\quad\text{ for every Lipschitz function }f\colon G\to\R.
\]
\end{definition}
We are finally in a position to state and prove the main result
of this section, which gives Theorem \ref{thm:Rademacher_property} as a corollary.
\begin{theorem}[Stability of the Rademacher property]\label{thm:Rademacher}
Let \(\big(\{G_i\}_{i\in I},\{\varphi_{ij}\}_{i\leq j}\big)\) be a direct
system of complete metric scalable groups (where the directed set \(I\) is countable) having direct limit \(\big(G,\{\varphi_i\}_{i\in I}\big)\)
in that category.
Let each space \(G_i\) be endowed with a family of null sets \(\mathcal N_i\)
such that \((G_i,\mathcal N_i)\) has the Rademacher property. Then \((G,\mathcal N)\)
has the Rademacher property, where \(\mathcal N\) is the family of null sets
defined in \eqref{eq:def_N}.
\end{theorem}
\begin{proof}
Let \(f\colon G\to\R\) be a fixed Lipschitz function. We define
\[
N_i\coloneqq\big\{p\in G\;\big|\;
e_i\in{\rm ND}(f\circ L_p\circ\varphi_i)\big\}\in\mathscr B(G)
\quad\text{ for every }i\in I.
\]
Then we claim that \(N\coloneqq\bigcup_{i\in I}N_i\in\mathcal N\). To prove it,
fix \(i\in I\) and \(q\in G\). We aim to show that
\begin{equation}\label{eq:Rademacher_claim}
\varphi_i^{-1}(q N_i)\subseteq{\rm ND}(f\circ L_q^{-1}\circ\varphi_i).
\end{equation}
Pick any \(x\in\varphi_i^{-1}(q N_i)\). This means that \(\varphi_i(x)=qp\) for
some \(p\in N_i\). By construction, we have \(e_i\in{\rm ND}(f\circ L_p\circ\varphi_i)\). Observe that for any \(g\in G_i\) and \(\lambda>0\) it holds
\[\begin{split}
\frac{(f\circ L_q^{-1}\circ\varphi_i)\big(x\,\delta^i_\lambda(g)\big)
-(f\circ L_q^{-1}\circ\varphi_i)(x)}{\lambda}
&=\frac{f\big(q^{-1}\,\varphi_i\big(x\,\delta^i_\lambda(g)\big)\big)
-f\big(q^{-1}\,\varphi_i(x)\big)}{\lambda}\\
&=\frac{f\big(p\,\varphi_i\big(\delta^i_\lambda(g)\big)\big)-f(p)}{\lambda}\\
&=\frac{(f\circ L_p\circ\varphi_i)\big(\delta^i_\lambda(g)\big)
-(f\circ L_p\circ\varphi_i)(e_i)}{\lambda},
\end{split}\]
so that \(e_i\in{\rm ND}(f\circ L_p\circ\varphi_i)\) is equivalent to
 \(x\in{\rm ND}(f\circ L_q^{-1}\circ\varphi_i)\). This proves
\eqref{eq:Rademacher_claim}. Given that the function \(f\circ L_q^{-1}\circ\varphi_i\)
is Lipschitz (as composition of Lipschitz maps) and \((G_i,\mathcal N_i)\) has
the Rademacher property, we conclude that \(\varphi_i^{-1}(q N_i)\in\mathcal N_i\)
for all \(i\in I\) and accordingly \(N\in\mathcal N\).

Let \(p\in G\setminus N\) be fixed. We aim to prove that \(f\) is G\^{a}teaux
differentiable at \(p\), which would be enough to achieve the statement.
Let us call \(G'\coloneqq\bigcup_{i\in I}\varphi_i(G_i)\), which is a dense
scalable subgroup of \(G\) (recall Remark \ref{rmk:density_claim}).

We define the function \(F_p\colon G'\to\R\) as
follows: given \(g\in G'\), we set
\[
F_p(g)\coloneqq\d_{e_i}(f\circ L_p\circ\varphi_i)(g_i)
\quad\text{ for any }i\in I\text{ and }g_i\in G_i\text{ such that }\varphi_i(g_i)=g.
\]
The well-posedness of \(F_p\) stems from the trivial identity
\begin{equation}\label{eq:Rademacher_aux}
\d_{e_i}(f\circ L_p\circ\varphi_i)(g_i)=\lim_{\lambda\to 0}
\frac{f\big(p\,\varphi_i\big(\delta^i_\lambda(g_i)\big)\big)-f(p)}{\lambda}=
\lim_{\lambda\to 0}\frac{f\big(p\,\delta_\lambda(g)\big)-f(p)}{\lambda}.
\end{equation}
It can be also readily checked -- by looking at \eqref{eq:Rademacher_aux} --
that \(F_p\) is a morphism of scalable groups.

Given any \(\lambda\in\R\setminus\{0\}\), we define the `incremental ratio' function
\({\rm IR}_\lambda\colon G\to\R\) as
\[
{\rm IR}_\lambda(g)\coloneqq\frac{f\big(p\,\delta_\lambda(g)\big)-f(p)}{\lambda}
\quad\text{ for every }g\in G.
\]
Observe that for any \(\lambda\in\R\setminus\{0\}\) and \(g,h\in G\) we have that
\[
\big|{\rm IR}_\lambda(g)-{\rm IR}_\lambda(h)\big|
=\frac{\big|f\big(p\,\delta_\lambda(g)\big)-f\big(p\,\delta_\lambda(h)\big)\big|}
{|\lambda|}\leq{\rm Lip}(f)\,
\frac{\sfd\big(p\,\delta_\lambda(g),p\,\delta_\lambda(h)\big)}{|\lambda|}
={\rm Lip}(f)\,\sfd(g,h),
\]
whence \(\{{\rm IR}_\lambda\}_{\lambda\neq 0}\) is an equiLipschitz family of
functions. As a consequence of \eqref{eq:Rademacher_aux}, we see that the
function \(F_p\colon G'\to\R\) is Lipschitz (being the pointwise limit of an
equiLipschitz family of functions), thus it can be uniquely extended to a
Lipschitz function \(\bar F_p\colon G\to\R\). By density of \(G'\) in \(G\) and
by continuity of the scalable group operations, we deduce that the extended
function \(\bar F_p\) is a Lipschitz morphism of scalable groups. Finally,
we also know that \({\rm IR}_\lambda\to\bar F_p\) in the
pointwise sense, so that \(f\) is G\^{a}teaux differentiable at \(p\) and
\(\d_p f=\bar F_p\). This completes the proof.
\end{proof}
\appendix
\section{Inverse limits of complete metric scalable groups}
\label{ss:IL_CMSG}
Let \(\big(\{G_i\}_{i\in I},\{P_{ij}\}_{i\leq j}\big)\) be an inverse system
of complete metric scalable groups. We denote by \(\delta^i\) and \(\sfd_i\)
the dilation and the distance on \(G_i\), respectively. Since
\(\big(\{G_i\}_{i\in I},\{P_{ij}\}_{i\leq j}\big)\) is an
inverse system in the category of scalable groups, it admits an inverse limit
\(\big((G',\delta'),\{P'_i\}_{i\in I}\big)\) in this category by Theorem
\ref{thm:IL_scalable_gps}. Set
\begin{equation}\label{eq:sup-metric}
\sfd'(x,y)\coloneqq\sup_{i\in I}\sfd_i\big(P'_i(x),P'_i(y)\big)
=\lim_{i\in I}\sfd_i\big(P'_i(x),P'_i(y)\big)
\quad\text{ for every }x,y\in G'.
\end{equation}
Then \(\sfd'\) is an \emph{extended} distance on \(G'\),
meaning that it satisfies the distance axioms, but
possibly taking value \(+\infty\). It can be readily checked
that \(\sfd'\) is \emph{compatible}, i.e.,
\[\begin{split}
\sfd'(xy,xz)=\sfd'(y,z),&\quad\text{ for every }x,y,z\in G',\\
\sfd'\big(\delta'_\lambda(x),\delta'_\lambda(y)\big)
=|\lambda|\,\sfd'(x,y),&\quad\text{ for every }\lambda\in\R
\text{ and }x,y\in G',
\end{split}\]
where in the second identity we are adopting the convention
that \(0\cdot\infty=0\). Moreover, we set
\begin{equation}\label{eq:def_IL}
G\coloneqq\big\{x\in G'\;\big|\;\sfd'(x,e)<+\infty\big\},
\qquad\delta\coloneqq\delta'|_{\R\times G}\colon\R\times G\to G,
\qquad\sfd\coloneqq\sfd'|_{G\times G},
\end{equation}
where \(e\) is the identity element of \(G'\).
Standard verifications show that \(G\) is a subgroup of \(G'\) and
\(\delta'_\lambda(G)\subseteq G\) for every \(\lambda\in\R\), so that
\(\delta\) is well-defined and is a dilation on \(G\). Moreover, the
restricted distance \(\sfd\) is compatible. Let us also define
\begin{equation}\label{eq:def_IL_morph}
P_i\coloneqq P'_i|_G\colon G\to G_i\quad\text{ for every }i\in I.
\end{equation}
It follows from the definition of \(\sfd'\) that each scalable group
morphism \(P_i\colon G\to G_i\) is \(1\)-Lipschitz.
\begin{definition}[Non-degenerate inverse system of MSGs]
\label{def:non-deg_IS}
Let \(\big(\{G_i\}_{i\in I},\{P_{ij}\}_{i\leq j}\big)\) be an inverse
system of metric scalable groups. Define \(\big((G,\delta,\sfd),
\{P_i\}_{i\in I}\big)\) as in \eqref{eq:def_IL} and \eqref{eq:def_IL_morph}.
Then we say that \(\big(\{G_i\}_{i\in I},\{P_{ij}\}_{i\leq j}\big)\) is
\emph{non-degenerate} provided the following conditions
are satisfied:
\begin{subequations}\begin{align}\label{eq:delta_non-deg_IS}
R_x\text{ is continuous at }e,&\quad\text{ for every }x\in G,\\
\label{eq:inv_non-deg_IS}
\delta(\cdot,x)\text{ is continuous at }t,&\quad\text{ for every }
x\in G\text{ and }t\in\R. 
\end{align}\end{subequations}
\end{definition}
Similarly to Propositions \ref{prop:char_non-def_DS} and
\ref{prop:char_non-def_DS_complete}, we have an alternative
(and more explicit) characterization of non-degeneracy.
We omit its standard proof.
\begin{proposition}\label{prop:char_non-def_IS}
Let \(\big(\{G_i\}_{i\in I},\{P_{ij}\}_{i\leq j}\big)\)
be an inverse system of metric scalable groups. Let us define
\(\big((G,\delta,\sfd), \{P_i\}_{i\in I}\big)\) as in \eqref{eq:def_IL}
and \eqref{eq:def_IL_morph}. Then \eqref{eq:delta_non-deg_IS} and
\eqref{eq:inv_non-deg_IS} are equivalent to
\begin{subequations}\begin{align}
\label{eq:delta_non-deg_IL_prop}
\inf_{\eta>0}\;\lim_{i\in I}\;\sup\bigg\{\sfd_i\big(P_i(yx),P_i(x)\big)\;
\bigg|\;y\in G,\,\lim_{j\in I}\sfd_j\big(P_j(y),e_j\big)<\eta\bigg\}=0,&
\quad\forall x\in G,\\
\label{eq:inv_non-deg_IL_prop}
\inf_{\eta>0}\;\lim_{i\in I}\;\sup_{\substack{s\in\R:\\|s-t|<\eta}}
\sfd_i\big(\delta^i_s(P_i(x)),\delta^i_t(P_i(x))\big)=0,&
\quad\forall x\in G,\,t\in\R,
\end{align}\end{subequations}
respectively.
\end{proposition}
\begin{lemma}\label{lem:IL_constr}
Let \(\big(\{G_i\}_{i\in I},\{P_{ij}\}_{i\leq j}\big)\) be
a non-degenerate inverse system of metric scalable groups.
Then \((G,\delta,\sfd)\) defined in \eqref{eq:def_IL} is
a metric scalable group and each map \(P_i\colon G\to G_i\) as in
\eqref{eq:def_IL_morph} is a morphism of metric scalable groups.
Moreover, if in addition the spaces \(\{G_i\}_{i\in I}\) are complete metric
scalable groups, then \(G\) is a complete metric scalable group as well.
\end{lemma}
\begin{proof}
The first part of statement immediately follows from Lemma
\ref{lem:cont_operations} and the very definition of non-degenerate
inverse system. In order to prove the second part of the statement,
suppose the distances \(\{\sfd_i\}_{i\in I}\) are complete. We aim to show
that \(\sfd\) is a complete distance. Fix a \(\sfd\)-Cauchy
sequence \((x^n)_n\subseteq G\).
Since \(P_i\) is Lipschitz, we deduce that \(\big(P_i(x^n)\big)_n\subseteq G_i\)
is a \(\sfd_i\)-Cauchy sequence for any \(i\in I\). Since \(\sfd_i\) is complete,
there is \(x_i\in G_i\) such that \(\lim_n P_i(x^n)=x_i\). Given any
\(i,j\in I\) with \(i\leq j\), we see that \(P_{ij}(y_j)=y_i\) by
letting \(n\to\infty\) in \(P_{ij}\big(P_j(x^n)\big)=P_i(x^n)\)
(here the continuity of \(P_{ij}\) plays a role). Consequently, there exists
a unique element \(x\in G'\) such that \(P'_i(x)=x_i\) for all \(i\in I\).
We claim that \(x\in G\) and \(\lim_n x^n=x\). To prove it, fix \(\eps>0\)
and choose \(\bar n\in\N\) such that \(\sfd(x^n,x^m)\leq\eps\) for every
\(n,m\geq\bar n\). This implies that \(\sfd_i\big(P_i(x^n),P_i(x^m)\big)\leq\eps\)
for every \(i\in I\) and \(n,m\geq\bar n\). By letting \(m\to\infty\) we
infer that \(\sfd_i\big(P_i(x^n),x_i\big)\leq\eps\) for every \(i\in I\)
and \(n\geq\bar n\), whence \(\sfd'(x^n,x)\leq\eps\) for all \(n\geq\bar n\).
This gives \(\sfd'(x,e)\leq\sfd'(x,x^{\bar n})+\sfd(x^{\bar n},e)<+\infty\),
so accordingly \(x\in G\). Moreover, it shows that \(\lim_n\sfd(x^n,x)=0\), as required.
\end{proof}
We can now state and prove our main results about inverse limits of MSGs
and CMSGs.
\begin{theorem}[Inverse limits of MSGs]\label{thm:IL_MSG}
Let \(\big(\{G_i\}_{i\in I},\{P_{ij}\}_{i\leq j}\big)\) be
an inverse system of metric scalable groups. Then the
following conditions are equivalent:
\begin{itemize}
\item[\(\rm i)\)] The inverse system
\(\big(\{G_i\}_{i\in I},\{P_{ij}\}_{i\leq j}\big)\) is non-degenerate
(in the sense of Definition \ref{def:non-deg_IS}).
\item[\(\rm ii)\)] The inverse limit of
\(\big(\{G_i\}_{i\in I},\{P_{ij}\}_{i\leq j}\big)\)
in the category of metric scalable groups exists.
Moreover, it coincides with the couple
\(\big((G,\delta,\sfd),\{P_i\}_{i\in I}\big)\)
given by \eqref{eq:def_IL} and \eqref{eq:def_IL_morph}.
\end{itemize}
\end{theorem}
\begin{proof}
{\color{blue}\({\rm i)}\Longrightarrow{\rm ii)}\)} Suppose i) holds.
In light of Lemma \ref{lem:IL_constr}, to prove that
\(\big(G,\{P_i\}_{i\in I}\big)\) is the inverse limit
of \(\big(\{G_i\}_{i\in I},\{P_{ij}\}_{i\leq j}\big)\),
it suffices to show that \(\big(G,\{P_i\}_{i\in I}\big)\)
satisfies the universal property. Fix a MSG \(H\) and a family
\(\{Q_i\}_{i\in I}\) of MSG morphisms \(Q_i\colon H\to G_i\)
such that \(Q_i=P_{ij}\circ Q_j\) for every \(i,j\in I\) with \(i\leq j\).
In particular, each \(Q_i\) is a morphism of scalable groups, then there
exists a unique scalable group morphism \(\Phi\colon H\to G'\) such that
\(Q_i=P'_i\circ\Phi\) for every \(i\in I\). Let us prove that
\(\Phi(H)\subseteq G\): given any \(y\in H\), we have that
\[
\sfd_i\big(P'_i(\Phi(y)),e\big)=\sfd_i\big(Q_i(y),e_i\big)
\leq\sfd_H(y,e_H)\quad\text{ for every }i\in I,
\]
whence \(\sfd'\big(\Phi(y),e\big)\leq\sfd_H(y,e_H)<+\infty\) and
accordingly \(\Phi(y)\in G\). The same computation shows that
\(\Phi\) is \(1\)-Lipschitz from \((H,\sfd_H)\) to \((G,\sfd)\).
Observe also that \(Q_i=P_i\circ\Phi\) for all \(i\in I\).
Therefore, the universal property is proven, thus obtaining ii).\\
{\color{blue}\({\rm ii)}\Longrightarrow{\rm i)}\)} Suppose ii) holds.
Then \eqref{eq:delta_non-deg_IS} and \eqref{eq:inv_non-deg_IS}
are satisfied, thus accordingly the inverse system
\(\big(\{G_i\}_{i\in I},\{P_{ij}\}_{i\leq j}\big)\)
is non-degenerate. This proves i).
\end{proof}
\begin{corollary}[Inverse limits of CMSGs]\label{cor:IL_CMSG}
Let \(\big(\{G_i\}_{i\in I},\{P_{ij}\}_{i\leq j}\big)\) be
an inverse system of complete metric scalable groups. Then the
following conditions are equivalent:
\begin{itemize}
\item[\(\rm i)\)] The inverse system
\(\big(\{G_i\}_{i\in I},\{P_{ij}\}_{i\leq j}\big)\) is non-degenerate
(in the sense of Definition \ref{def:non-deg_IS}).
\item[\(\rm ii)\)] The inverse limit of
\(\big(\{G_i\}_{i\in I},\{P_{ij}\}_{i\leq j}\big)\)
in the category of complete metric scalable groups exists.
Moreover, it coincides with the couple
\(\big((G,\delta,\sfd),\{P_i\}_{i\in I}\big)\)
given by \eqref{eq:def_IL} and \eqref{eq:def_IL_morph}.
\end{itemize}
\end{corollary}
\begin{proof}
The claim follows from Theorem \ref{thm:IL_MSG} and the last part
of the statement of Lemma \ref{lem:IL_constr}.
\end{proof}

\textbf{Acknowledgements.}
Both authors were supported by the European Research Council
(ERC Starting Grant 713998 GeoMeG Geometry of Metric Groups).
The first named author was also supported by the Academy of Finland
(grant 288501 Geometry of subRiemannian groups and grant 322898
Sub-Riemannian Geometry via Metric-geometry and Lie-group Theory).
The second named author was also supported by the Academy of Finland
(project number 314789) and by the Balzan project led by Prof.\ Luigi
Ambrosio.

\begin{thebibliography}{LDLM19}

\bibitem[Aro76]{Aronszajn1976}
Nachman Aronszajn.
\newblock Differentiability of {L}ipschitzian mappings between {B}anach spaces.
\newblock {\em Studia Mathematica}, 57(2):147--190, 1976.

\bibitem[BGM13]{BGT13}
Fabrice Baudoin, Maria Gordina, and Tai Melcher.
\newblock Quasi-invariance for heat kernel measures on sub-{R}iemannian
  infinite-dimensional {H}eisenberg groups.
\newblock {\em Trans. Amer. Math. Soc.}, 365(8):4313--4350, 2013.

\bibitem[DG08]{DriverGordina08}
Bruce~K. Driver and Maria Gordina.
\newblock Heat kernel analysis on infinite-dimensional {H}eisenberg groups.
\newblock {\em J. Funct. Anal.}, 255(9):2395--2461, 2008.

\bibitem[DZ19]{LDZ19}
Enrico~Le Donne and Roger Z\"ust.
\newblock Space of signatures as inverse limits of {C}arnot groups.
\newblock {\em ArXiv e-prints}, 2019.

\bibitem[GMV15]{GMV15}
Erlend Grong, Irina Markina, and Alexander Vasil'ev.
\newblock Sub-{R}iemannian geometry on infinite-dimensional manifolds.
\newblock {\em J. Geom. Anal.}, 25(4):2474--2515, 2015.

\bibitem[Lan84]{lang84}
Serge Lang.
\newblock {\em Algebra}.
\newblock Addison-Wesley Publishing Company, Advanced Book Program, Reading,
  MA, second edition, 1984.

\bibitem[LD13]{LD13}
Enrico Le~Donne.
\newblock A metric characterization of {C}arnot groups.
\newblock {\em Proceedings of the American Mathematical Society}, 143, 2013.

\bibitem[LD17]{LD17}
Enrico Le~Donne.
\newblock A primer on {C}arnot groups: homogeneous groups,
  {C}arnot-{C}arath\'{e}odory spaces, and regularity of their isometries.
\newblock {\em Anal. Geom. Metr. Spaces}, 5(1):116--137, 2017.

\bibitem[LDLM19]{LDLM19}
Enrico Le~Donne, Sean Li, and Terhi Moisala.
\newblock Infinite-{D}imensional {C}arnot {G}roups and {G}\^{a}teaux
  {D}ifferentiability.
\newblock {\em J Geom Anal}, 2019.
\newblock https://doi.org/10.1007/s12220-019-00324-x.

\bibitem[Mel09]{Mel09}
Tai Melcher.
\newblock Heat kernel analysis on semi-infinite {L}ie groups.
\newblock {\em J. Funct. Anal.}, 257(11):3552--3592, 2009.

\bibitem[ML98]{MacLane98}
Saunders Mac~Lane.
\newblock {\em Categories for the working mathematician}, volume~5 of {\em
  Graduate Texts in Mathematics}.
\newblock Springer-Verlag, New York, second edition, 1998.

\bibitem[Moi20]{M20}
Terhi Moisala.
\newblock {\em Unraveling {I}ntrinsic {G}eometry of {S}ets and {F}unctions in
  {C}arnot groups}.
\newblock PhD thesis, University of Jyv\"askyl\"a, 2020.

\bibitem[Mon36]{Mon36}
Deane Montgomery.
\newblock Continuity in topological groups.
\newblock {\em Bull. Amer. Math. Soc.}, 42(12):879--882, 1936.

\bibitem[MPS19]{MPS19}
Valentino Magnani, Andrea Pinamonti, and Gareth Speight.
\newblock Porosity and differentiability of {L}ipschitz maps from stratified
  groups to {B}anach homogeneous groups.
\newblock {\em Annali di Matematica Pura ed Applicata (1923-)}, pages 1--24,
  2019.

\bibitem[MR13]{mr13}
Valentino Magnani and Tapio Rajala.
\newblock Radon--{N}ikodym {P}roperty and {A}rea {F}ormula for {B}anach
  {H}omogeneous {G}roup {T}argets.
\newblock {\em International Mathematics Research Notices},
  2014(23):6399--6430, 2013.

\bibitem[MZ55]{MZ_1955}
Deane Montgomery and Leo Zippin.
\newblock {\em Topological transformation groups}.
\newblock Interscience Publishers, New York-London, 1955.

\bibitem[Pan89]{Pansu89}
Pierre Pansu.
\newblock M\'{e}triques de {C}arnot-{C}arath\'{e}odory et quasiisom\'{e}tries
  des espaces sym\'{e}triques de rang un.
\newblock {\em The Annals of Mathematics}, 129(1):1--60, 1989.

\bibitem[Tka14]{Tkacenko14}
Mikhail Tkachenko.
\newblock Paratopological and semitopological groups versus topological groups.
\newblock In {\em Recent progress in general topology. {III}}, pages 825--882.
  Atlantis Press, Paris, 2014.

\end{thebibliography}
\end{document}